\documentclass[11pt]{article}

\usepackage[left=30mm,right=30mm,top=30mm,bottom=30mm]{geometry}
\usepackage{graphicx,color}
\usepackage[utf8]{inputenc}
\usepackage{amsmath}
\usepackage{amssymb}
\usepackage{amsthm}
\usepackage{comment}
\usepackage{enumerate}
\usepackage{mathtools}
\usepackage{cases}
\usepackage{multicol}
\usepackage{tabularx}
\usepackage{todonotes}
\usepackage{dsfont}

\theoremstyle{plain}
\newtheorem{thm}{Theorem}[section]
\newtheorem{prop}[thm]{Proposition}
\newtheorem{lem}[thm]{Lemma}

\theoremstyle{definition}
\newtheorem{defi}[thm]{Definition}
\newtheorem{ex}[thm]{Example}

\theoremstyle{remark}
\newtheorem{remark}[thm]{Remark}

\newcommand{\bem}{\begin{bmatrix}}
\newcommand{\enm}{\end{bmatrix}}
\newcommand{\la}{\langle}
\newcommand{\ra}{\rangle}
\newcommand{\Ha}{\mathcal{H}}
\newcommand{\z}{\boldsymbol{z}}

\newcommand{\x}{\xi}

\newcommand{\dt}[1]{\frac{\partial #1}{\partial t}}
\newcommand{\dz}[1]{\frac{\partial #1}{\partial \x}}

\newcommand{\C}{\mathbb C}
\newcommand{\R}{\mathbb R}
\newcommand{\N}{\mathbb N}
\newcommand{\X}{\mathcal{X}}
\newcommand{\Y}{\mathcal{Y}}
\newcommand{\Z}{\mathcal{Z}}

\newcommand{\K}{\mathbb K}
\renewcommand{\L}{\mathcal{L}}
\newcommand{\dom}{\mathcal{D}}
\newcommand{\V}{\mathcal{V}}
\newcommand{\W}{\mathcal{W}}
\newcommand{\setdef}[2]{\left\{ #1 \left\vert\vphantom{#1} #2 \right.\right\}}

\newcommand{\e}{\mathrm{e}}
\newcommand{\dx}[1][x]{\,\mathrm{d}#1}

\newcommand{\Real}{{\rm Re}\,}

\renewcommand{\ker}{{\rm ker}\,}

\newcommand{\inv}{^{-1}}

\newcommand{\pres}{p_{\rm res \vphantom{d}}} 

\title{Index concepts for linear differential-algebraic equations in finite and infinite dimensions}
\author{Mehmet Erbay, Birgit Jacob, Kirsten Morris, Timo Reis and Caren Tischendorf}
\begin{document}

\maketitle

\begin{abstract}
Different index concepts for linear differential-algebraic equations are defined in the general Banach space setting,  and compared. For regular finite-dimensional linear   differen\-tial-algebraic equations, all these indices  exist and are equivalent. For infinite-dimensional systems, the situation is more complex. It is proven that although some indices imply others, in general they are not equivalent. The situation is illustrated with a number of examples.
\end{abstract}

\section{Introduction}

In this article we  take a closer look at the index terms for infinite-dimensional differential-algebraic systems (DAE) of the form
    \begin{equation}\label{eqn:dae}
        \frac{{\rm d}}{{\rm d}t}Ex(t)=Ax(t)+f(t), \quad t\geq 0,
    \end{equation}
where $E\colon \X\to\Z$ is a bounded linear operator (denoted by $E\in \L(\X,\Z)$), $(A,\dom(A))$ is a closed and densely defined linear operator from $\X$ to $\Z$ and $f\colon[0,\infty) \to\Z$. Throughout this article, $\X$ and $\Z$ are  Banach spaces and the DAE \eqref{eqn:dae} is assumed to be regular. That is,
    \begin{equation*}
        \rho(E,A)\coloneqq \left\{\lambda \in \C \,\middle|\, (sE-A)\inv \in \L(\Z,\X)\right\}\not=\emptyset.
    \end{equation*}
 By a \textit{solution} of \eqref{eqn:dae} we mean a \textit{classical solution}, that is, a function $x\colon [0,\infty)\rightarrow \dom(A)$ such that $Ex(\cdot)$ is continuously differentiable as a~function with values in $\Z$, and \eqref{eqn:dae} is satisfied for every $t\ge 0$.

The index of a DAE can be defined in a number of various ways. Examples include the differentiation index, the nilpotency index, the resolvent index and the radiality index \cite{GernandHallerReis, Kunkelmehrmann, fedorov-sviridyuk, STrostorff}. Not all indices are defined in the infinite-dimensional case. For instance, the nilpotency index of a DAE demands a Weierstraß form (defined formally below), which is not always available. 

Our aim in writing this paper is to 'collect' all the index terms that are applicable in the infinite-dimensional case, and to characterize and compare them to each other. In particular, we investigate the resolvent index $\pres^{(E,A)}$, the chain index $p_{\rm chain}^{(E,A)}$, the radiality index $p_{\rm rad}^{(E,A)}$, the nilpotency index $p_{\rm nilp}^{(E,A)}$, the differentiation index $p_{\rm diff}^{(E,A)}$ and the perturbation index $p_{\rm pert}^{(E,A)}$. Several of these indices have not previously been defined for infinite-dimensional systems.

One of our main results is  that if  all the  indices mentioned in the previous paragraph exist then
    \begin{align*}
        p_{\rm rad}^{(E,A)}+1 \geq \pres^{(E,A)} \geq p_{\rm nilp}^{(E,A)} = p_{\rm diff}^{(E,A)} =p_{\rm chain}^{(E,A)}.
    \end{align*}
    If in addition the operator $A_1$ in the Weierstraß form generates a $C_0$-semigroup, then
  \begin{align*}
        p_{\rm rad}^{(E,A)}+1 \geq \pres^{(E,A)} \geq p_{\rm nilp}^{(E,A)} = p_{\rm diff}^{(E,A)} =p_{\rm chain}^{(E,A)}=p_{\rm pert}^{(E,A)}.
    \end{align*}
Furthermore, Proposition \ref{nilp>=rad+1}  implies that in the finite-dimensional case, equality holds in all these bounds.

We conclude the introduction with some notation.
 For a $\lambda \in \rho(E,A)$ we call $(\lambda E-A)\inv$, $R^E(\lambda,A)\coloneqq (\lambda E-A)\inv E$ and $L^E(\lambda,A)\coloneqq E(\lambda E-A)\inv$ the \textit{resolvent}, \textit{right-$E$ resolvent} and \textit{left-$E$ resolvent} of $A$ respectively.

\section{Weierstra\ss{} form}
        Consider differential-algebraic systems
        of the form \eqref{eqn:dae}. 
        Let  also $\tilde\X, \tilde\Z$ be  Banach spaces, $\tilde E\in\L(\tilde\X,\tilde \Z)$ and $\tilde A\colon \dom(\tilde A)\subseteq \tilde \X\to\tilde\Z$ closed and densely defined.

    \begin{defi}\label{def:equiv}
        Two differential-algebraic systems $\frac{{\rm d}}{{\rm d}t}Ex=Ax$ and $\frac{{\rm d}}{{\rm d}t}\tilde Ex=\tilde Ax$  are \textit{equivalent}, denoted by $(E, A) \sim (\tilde E, \tilde A)$, if there are two bounded isomorphisms $P\colon\X\to\tilde\X$, $Q\colon\Z\to\tilde\Z$, such that $E=Q\inv \tilde E P$ and $A=Q\inv \tilde A P$.
    \end{defi}

    \begin{defi}
        A bounded operator $N\in\L(\X)$ is called {\em nilpotent}, if there exists a $p\in \N$, such that $N^l\neq 0$ for all $l<p$ and $N^p=0$. $p$ is called the {\em degree of nilpotency}.
    \end{defi}

This definition may be slightly different  in other references. For example, in \cite{fedorov-sviridyuk} the degree of nilpotency is $p-1$ and not $p$.

    \begin{defi}
        The DAE \eqref{eqn:dae} has a \textit{Weierstraß form}, if there exists a Hilbert space $\Y=\Y^1\oplus\Y^2$, such that
        \begin{equation}\label{eqn:form-inf}
            (E,A)\sim \left(\begin{bmatrix}I_{\Y^1} & 0\\0 & N\end{bmatrix},\begin{bmatrix}A_1 & 0\\0 & I_{\Y^2}\end{bmatrix}\right),
        \end{equation}
        where $N\colon\Y^2\to\Y^2$ is a bounded linear nilpotent operator, $A_1 \colon D(A_1) \subseteq\Y^1\to\Y^1$ is a linear operator  and $I_{\Y^i}$ indicates the identity operator on the associated subspace $\Y^i$, $i=1,2$.
    \end{defi}

This form is also known variously as the \textit{quasi-Weierstraß form} \cite{BergerIlchmannTrenn2012} or \textit{Weierstraß canonical form} \cite{Kunkelmehrmann}. In finite dimensions the operator $A_1$ is generally a  Jordan matrix $J$. In this case, the Weierstraß form is unique up to isomorphisms and therefore the nilpotency degree of  $N$ is uniquely determined. To be more precise, assume that $(E,A)$ has two different Weierstraß forms $\left(\big[\begin{smallmatrix}I & 0\\ 0 & N_i\end{smallmatrix}\big], \big[\begin{smallmatrix}J_i & 0\\0&I\end{smallmatrix}\big]\right)$, $i=1,2$. Then the sizes of the Jordan blocks $J_1, J_2$ and of the nilpotent operators $N_1, N_2$  coincide, as well as the degree of nilpotency of these nilpotent operators \cite[Lem.~2.10]{Kunkelmehrmann}.

In the next few sections we will go through a variety of different index term and generalise/adapt them for the infinite-dimensional case. Most of the terms  are already known in the finite-dimensional case.

\section{Resolvent index}
The \textit{resolvent index} has already been defined in \cite[p.~5]{GernandHallerReis}, \cite[p.~8]{STrostorff} and \cite[ch.~6.1]{gernandt_pseudo-resolvent_2023}. It has the advantage that it does not require a Weierstraß form. Thus, this definition can be easily extended to the infinite-dimensional case. The only difficulty encountered in calculating this index is the calculation of the resolvent and its growth rate, which is  a greater hurdle in the infinite-dimensional case.

    \begin{defi}[\textbf{resolvent index}]\hfill\\
        The \textit{resolvent index} of $(E,A)$ is the smallest integer  $\pres^{(E,A)}\in \N_0$, such that there exists a $\omega \in \R$, $C>0$ with $(\omega, \infty)\subseteq \rho(E,A)$ and
        \begin{equation}\label{def:resolvent-index}
            \left\Vert(\lambda E-A)\inv\right\Vert\leq C \left|\lambda\right|^{\pres^{(E,A)}-1}
        \end{equation}
        for all $\lambda\in (\omega,\infty)$. The resolvent index is called a \textit{complex resolvent index}, denoted by $p_{\rm c, res}^{(E,A)}\in \N_0$, if $\C_{\Real >\omega}\subseteq \rho(E,A)$ and \eqref{def:resolvent-index} holds for $p_{\rm c, res}^{(E,A)}$.
    \end{defi}

Note that the resolvent index can also defined in a weaker form as seen in \cite[ch.~5\& 6]{gernandt_pseudo-resolvent_2023}. Clearly, $\pres^{(E,A)}\le p_{\rm c,res}^{(E,A)}$. The next proposition shows that this index is uniquely defined.

    \begin{prop}\label{res-equiv}
        The resolvent index, given that it exists,  is uniquely defined. To be more precise, let $(E,A) \sim (\tilde E, \tilde A)$. Then $p_{\mathrm{res}}^{(E,A)}=p_{\mathrm{res}}^{(\tilde E,\tilde A)}$ and $p_{\mathrm{c,res}}^{(E,A)}=p_{\mathrm{c,res}}^{(\tilde E,\tilde A)}$.
    \end{prop}

    \begin{proof}
        Since $(E,A)\sim (\tilde E, \tilde A )$ there exists two isomorphisms $P\colon \X \to \tilde \X$, $Q\colon \Z \to \tilde \Z$ with
        \begin{equation*}
            E = Q\inv \tilde E P \quad \text{and} \quad A = Q\inv \tilde A P.
        \end{equation*}
        Assume  that $(E,A)$ has resolvent index $\pres^{(E,A)}$. That is,  there exists a $C>0$, $\omega \in\R$, such that $(\omega,\infty)\subseteq \rho(E,A)$ and
        \begin{equation*}
            \Vert (\lambda E-A)\inv \Vert \leq C \left\vert\lambda\right\vert^{\pres^{(E,A)} -1}, \quad \lambda >\omega.
        \end{equation*}
        Since
        \begin{equation*}
            (\lambda \tilde E-\tilde A)\inv \tilde E= P (\lambda E-A)\inv E Q\inv, \quad \lambda >\omega
        \end{equation*}
        it follows that  $(\omega,\infty)\subseteq \rho(E,A)\subseteq\rho(\tilde E,\tilde A)$ and
        \begin{equation*}
            \Vert (\lambda \tilde E-\tilde A)\inv \Vert \leq \left\Vert P \right\Vert \left\Vert (\lambda E-A)\inv \right\Vert \left\Vert Q\inv \right\Vert \leq  \tilde C \left\vert\lambda\right\vert^{\pres^{(E,A)} -1}, \quad \lambda >\omega,
        \end{equation*}
        for $\tilde C \coloneqq C\left\Vert P \right\Vert \left\Vert Q\inv \right\Vert$. Thus, the resolvent index of $(\tilde E, \tilde A)$ is at most $\pres^{(E,A)}$. The other estimate follows from an  equivalent argument and switching $(E,A)$ and $(\tilde E, \tilde A)$. The statement concerning the complex resolvent index follows similarly.
    \end{proof}

    Next, we will show the existence of the (complex) resolvent index for a special class of systems, namely $\X=\Z$, $E$ is non-negative and $A$ is dissipative. Note that we call $E$ non-negative, denoted by $E\geq 0$, if $\langle Ex,x\rangle\geq 0$ for all $x\in\X$ and we call $A$ dissipative, if $\Real \langle Ax,x\rangle\leq 0$ for all $x\in\dom(A)$. Such systems, are known as \textit{port-Hamiltonian DAEs} or \textit{abstract dissipative DAEs} (see \cite[ch.~ 7]{gernandt_pseudo-resolvent_2023} and \cite{mehrmann2023abstract}).

    \begin{thm}\label{res-index-bound}
        Let $\X=\Z$, $E\in \L(\X)$ be non-negative self-adjoint and $A\colon\dom(A)\subseteq\X\to\X$ be dissipative. If there exists a $\omega>0$, such that $(\omega,\infty) \subseteq \rho(E,A)$,  then $p_{\mathrm{res}}^{(E,A)}\le 2$. If also  $\C_{\Real >\omega}\subseteq \rho(E,A)$, then $p_{\mathrm{c,res}}^{(E,A)}\le 3$.
    \end{thm}

    \begin{proof}
        For any $x\in\X$ define  $z=(\lambda E-A)\inv x .$ By using the dissipativity of $A$ and the positivity of $E$ we deduce that
        \begin{align*}
            \Real \langle (\lambda E-A)\inv x,x\rangle = \langle z, Ez\rangle \Real \lambda - \Real \langle z, A z\rangle \geq 0
        \end{align*}
        for all $\lambda\in \rho(E,A)\cap\C_{\Real >0}$, . Thus, for every $x\in \X$ the function $$f_x\colon \lambda\mapsto\langle (\lambda E-A)\inv x,x\rangle + \left\Vert x \right\Vert^2$$ is positive real. Let $\mu>\omega$. Using \cite[Thm.~3]{goldberg} implies
        \begin{align*}
            \left\vert f_x(\lambda)\right\vert
            \leq \left\vert f_x(\mu)\right\vert \frac{\left\vert \lambda\right\vert^2 + \left\vert \mu\right\vert^2 + 3 \left\vert \lambda \mu \right\vert}{\mu \Real \lambda}
            \leq \left\vert f_x(\mu)\right\vert \frac{5}{\mu}\frac{\left\vert\lambda\right\vert^2}{\Real \lambda}, \quad \lambda \in \C_{\Real>\mu}.
        \end{align*}
        Hence, together with the Riesz representation theorem we derive
        \begin{align*}
            \left\Vert (\lambda E-A)\inv \right\Vert
            = \sup_{\substack{x,y\in \X \\ \Vert x\Vert=\Vert y\Vert = 1}} \left\vert \langle (\lambda E-A)\inv x,y\rangle \right\vert
            \leq \sup_{\substack{x\in\X \\ \left\Vert x \right\Vert=1}} 2\left\vert f_x(\mu)\right\vert
            \leq K\frac{\left\vert\lambda\right\vert^2}{\Real \lambda},
        \end{align*}
        for all $\lambda \in \C_{\Real >\omega}$ with $K=(\left\Vert (\mu E-A)\inv \right\Vert +1)\frac{10}{\mu}$. Thus, the complex resolvent index is at most $3$ and if $\lambda>\omega$ is real, then $\frac{\left\vert\lambda\right\vert^2}{\Real \lambda}= \lambda$ and the resolvent index is at most $2$.
    \end{proof}

    \begin{ex} {\bf ($\pres^{(E,A)} = 2$ and $p_{\mathrm{c, res}\vphantom{d}}^{(E,A)}=3$.)}
        Define $A={\rm diag\,} (A_0, A_1, A_2,\ldots)$ with
        \begin{align*}
            A_0=\bem 0 & -1\\ 1& 0\enm,\quad A_k=\bem 0 & \sqrt{k^4+1} \\ -\sqrt{k^4+1} & -2\enm, \quad k\in \N,
        \end{align*}
        and $\dom(A)\coloneqq \{x\in \ell^2 \mid Ax\in\ell^2\}$. Then $A$ can be extended to $\ell^2$, which will denoted by $A_{-1}$. Define $E \in\L(\ell^2)$, $B\colon\R\to \dom(A^\ast)'$ and $C\colon \dom(A)\to\R$ with $E={\rm diag\,}(E_0, E_1, E_2,\ldots)$ and $B=(B_0, B_1, B_2, \ldots)^T = C^\ast$, whereby
        \begin{align*}
            & & E_0 &= \bem 1 & 0\\ 0 & 0\enm, & E_k &= \bem 1 & 0 \\ 0 & 1\enm, \quad k\in\N, & & \\
            & & B_0 &= \bem 0\\ 1\enm, & B_k &=\bem 0\\ k^{\frac{5}{4}}\enm, \quad k\in\N. & &
        \end{align*}
        Consider the following system
        \begin{align*}
            \frac{{\rm d}}{{\rm d}t}\underbrace{\bem E & 0 & 0\\ 0 & 0 & 0\\ 0 & 0 & 0\enm}_{\mathcal{E}} \bem x_1\\ x_2\\ x_3\enm = \underbrace{\bem A_{-1} & B & 0\\ -C & 0 & I\\ 0 & -I & 0\enm}_{\mathcal{A}} \bem x_1\\ x_2 \\ x_3 \enm
        \end{align*}
        on $\X=\ell^2\times \R\times \R$. Obviously, $\mathcal{E}$ is non-negative and self-adjoint and, by its construction, $\mathcal{A}$ with maximal domain  is dissipative. Thus, $(\mathcal{E}, \mathcal{A})$ satisfies the conditions of Theorem \ref{res-index-bound}.

        It will now be shown that $\pres^{(\mathcal{E}, \mathcal{A})} = 2$ and $p_{\rm c,res}^{(\mathcal{E}, \mathcal{A})}=3$. For $s\in\rho(\mathcal{E}, \mathcal{A})$
        \begin{equation*}
            (s\mathcal{E}-\mathcal{A})\inv = \bem (sE-A)\inv & 0 & (sE-A)\inv B \\ 0 & 0 & I \\ -C(sE-A)\inv & -I & C(sE-A)\inv B\enm.
        \end{equation*}

        It can be shown that
        \begin{equation*}
            G(s)\coloneqq C(sE-A)\inv B =  s+\sum_{k=1}^\infty \frac{k^{\frac{5}{2}}s}{s^2+2s+k^4+1},
        \end{equation*}
a function that George Weiss (Tel Aviv) scribbled on paper for one of the authors. Thanks for that!\\
        Since $G(s)\geq s$ for $s\geq 0$, the resolvent $(s\mathcal{E}-\mathcal{A})\inv$ grows at least linearly  along the real axis and together with Theorem \ref{res-index-bound}  this implies $\pres^{(\mathcal{E}, \mathcal{A})}=2$.

        Now consider growth in the entire right-hand-plane. Letting $\sigma>0$ and $s_n\coloneqq \sigma + i n^2$ for $n\in\N$,
        \begin{align*}
            \Real G(s_n) &= \sigma + \sum_{k=1}^\infty \frac{k^{\frac{5}{2}}(\sigma ((1+\sigma)^2+k^4)+(2+\sigma)n^4)}{((1+\sigma)^2+(k^4-n^4))^2+4(1+\sigma)^2n^4 }\\
            &\geq \sigma + \frac{n^{\frac{5}{2}}(\sigma ((1+\sigma)^2+n^4)+(2+\sigma)n^4)}{(1+\sigma)^4+4(1+\sigma)^4n^4 }\\
            &\geq \frac{2n^{\frac{5}{2}}n^{4}}{(1+\sigma)^4(1+4n^4)} = \frac{2n^{\frac{5}{2}}}{5(1+\sigma)^4} \frac{5n^4}{1+4n^4}\\
            &\geq \frac{2n^{\frac{5}{2}}}{5(1+\sigma)^4}, \quad n\in\N.
        \end{align*}
        By choosing an $c>0$ and $N\in\N$ such that $\vert s_n\vert^{\frac{5}{4}} \leq c \frac{2n^{\frac{5}{2}}}{5(1+\sigma)^4}$ for every $n\geq N$, one obtains
        \begin{equation*}
            \left\vert G(s_n)\right\vert \geq \Real G(s_n) \geq \frac{1}{c}\left\vert s_n\right\vert^{\frac{5}{4}}, \quad n\geq N.
        \end{equation*}
        Hence, the growth of $G(s_n)$ is more than quadratic  along lines $\Real s = \sigma$ parallel to the imaginary axis for every $\sigma>0$. Therefore, using Theorem \ref{res-index-bound} it follows that  $p_{\mathrm{c,res}}^{(E,A)}=3$.
    \end{ex}

    \begin{ex} {\bf (A class of well-posed systems with $\pres^{(E,A)} \le 2$)} \label{never-ending-example}
        Let $A_o$ with $\dom (A_o) \subset \W$ generate a $C_0$-semigroup on  $\W$, where $\W$ is a Hilbert space. We  indicate the growth bound of the semigroup by $\omega$. For $b\in \W$ and $c\in \dom (A_o^*)$ with $\langle b, c\rangle\not=0$ define the operators $B u = b u$ where $u \in \C$ and $ C z= \la z, c \ra $ for any $z \in \W$.  We define the DAE on $\Z= \W \times \C $ by
        \begin{equation}
            \frac{{\rm d}}{{\rm d}t} \underbrace{\bem I & 0 \\ 0 & 0 \enm}_{\eqqcolon E} x(t)= \underbrace{\bem A_o & B \\C & 0 \enm}_{\eqqcolon A} x(t) \quad t \geq 0.
            \label{eq-zero1}
        \end{equation}
	Defining for $s \in \rho (A_o ), $
        \begin{equation*}
            G(s)= \la (sI-A_o)^{-1} b ,c \ra
        \end{equation*}
        we obtain
        \begin{equation*}
            (sE-A)^{-1} = \bem (sI-A_o)^{-1} - (sI-A_o)^{-1} B G(s)^{-1} C (sI-A_o)^{-1} & (sI-A_o)^{-1} B G(s)^{-1} \\  G(s)^{-1} C (sI-A_o)^{-1} & -G(s)^{-1}\enm .
        \end{equation*}
        Note that the condition $\langle b, c\rangle\not=0$ implies $G(s)\not=0$ for $s \in \rho (A_o )$.
        Thus, the resolvent of $(E,A)$ is non-empty and the system is regular.
        Since
        \begin{equation*}
            \lim_{s \to \infty } s G(s) = \la c,b\ra,
        \end{equation*}
        for large $s$, $G(s)^{-1} \leq M s$ and so the resolvent index is at most $2$.
    \end{ex}

\section{Chain index}

    The concept of Wong sequences \cite{Wong1974} are needed to define the chain index. For $\X=\Z=\K^n$ ($n\in \N$) we define
    \begin{align*}
        & & \V_0 &\coloneqq \K^n,  &\V_{i+1}&\coloneqq A\inv (E(V_i)), & \quad i\in\N, & & \\
        & & \W_0 &\coloneqq \left\{0\right\},&  \W_{i+1}&\coloneqq E\inv(A(W_i)), & \quad i\in\N. & &
    \end{align*}
    Thus, there exists $k,l\in\N$ such that $\V_{i+1}\subsetneq \V_i$, $0\leq i < l$, $\W_j\subsetneq \W_{j+1}$, $0\leq j < k$ and $\V_l=\V_{l+r}$, $\W_k=\W_{l+r}$ for all $r\in\N$ \cite[p.~3]{BergerIlchmannTrenn2012}. Let $x_k\in\W_k$. Then, there exists a $x_{k-1}\in\W_{k-1}$ with $Ex_k=Ax_{k-1}$. Inductively, one obtains
    \begin{align}\label{chain}
        \begin{split}
            Ex_1 &= 0,\\
            Ex_2 &= Ax_1,\\
            & \; \;\vdots \\
            Ex_k &= Ax_{k-1},
        \end{split}
    \end{align}
    $x_r\in \W_{r}$, $0\leq r\leq k$, which is called a \textit{chain of length $k$}. The concept of such a chain is not new and can be found in \cite[p.~15]{BergerIlchmannTrenn2012} for the finite-dimensional case or \cite[p.~18]{fedorov-sviridyuk} for arbitrary Banach spaces. It should be mentioned that the latter reference  does not call \eqref{chain} a chain.

    \begin{defi}[\textbf{chain index}]\hfill\\
        We call $(x_1,\ldots, x_{p})\in\dom(A)^{p-1}\times \X$, $p\in\N$, a \textit{chain of} $(E,A)$ of length $p$, if
        $x_1\in\ker E\setminus\{0\}$, $x_k\notin \ker A$, $k=1,\ldots,p-1$, and
        \begin{equation}
            Ex_{k+1}=Ax_{k},\quad k=1,\ldots,p-1.
        \end{equation}
        The \textit{chain index} $p_{\rm chain}^{(E,A)}\in\N_0$ of $(E,A)$ is the supremum over all chain lengths of $(E,A)$.
    \end{defi}

    With this definition we exclude the case where the chain index can be infinite. Such an example is obtained when $(E,A)$ has a form as in \eqref{eqn:form-inf}, whereby $N$ is quasi-nilpotent. For such systems the chain index is not defined.

    \begin{prop}\label{chain-equiv}
        The chain index is, given that it exists, uniquely defined. To be more specific, let $(E,A) \sim (\tilde E, \tilde A)$. Then $p_{\mathrm{chain}}^{(E,A)}=p_{\mathrm{chain}}^{(\tilde E,\tilde A)}$.
    \end{prop}

    \begin{proof}
        Since $(E,A)\sim (\tilde E, \tilde A )$ there exists two isomorphisms $P\colon \X \to \tilde \X$, $Q\colon \Z \to \tilde \Z$, such that
        \begin{equation*}
            E = Q\inv \tilde E P \quad \text{and} \quad A = Q\inv \tilde A P.
        \end{equation*}
        Let $p_{\rm chain}^{(E,A)}\in\N$. Thus, there exists a chain $(x_1,\ldots,x_{p_{\rm chain}^{(E,A)}})$, such that
        \begin{equation*}
            Ex_1 = 0, Ex_2=Ax_1, \ldots, Ex_{p_{\rm chain}^{(E,A)}} = Ax_{p_{\rm chain}^{(E,A)}-1}.
        \end{equation*}
        Since $E=Q\inv \tilde E P$, $A=Q\inv \tilde A P$ and $Q$ is an isomorphism,
        \begin{equation*}
            \tilde E P x_1 = 0, \tilde E P x_2 = \tilde A P x_1, \ldots, \tilde E P x_{p_{\rm chain}^{(E,A)}} = \tilde A P x_{p_{\rm chain}^{(E,A)}-1}.
        \end{equation*}
        Because $P$ is an isomorphism,  $Px_1 \neq 0$ and thus $Px_1 \in \ker \tilde E\backslash\{0\}$. If we assume that $Px_k \in \ker \tilde A$ for any $1\leq k \leq p_{\rm chain}^{(E,A)}-1$, then $0=Q\inv \tilde A P x_k = Ax_k$, which would be a contradiction to $x_k\in \X\backslash\ker A$. Therefore, $Px_k\notin \ker \tilde A$ for all $1\leq k \leq p_{\rm chain}^{(E,A)}-1$. Hence, $(Px_1, \ldots, Px_{p_{\rm chain}^{(E,A)}})$ denotes a chain of $(\tilde E, \tilde A)$ of length $p_{\rm chain}^{(E,A)}$ . Therefore, the chain index of $(\tilde E, \tilde A)$ is bounded from below by $p_{\rm chain}^{(E,A)}$ and, equivalently, the chain index of $(E,A)$ is bounded from below by $p_{\rm chain}^{(\tilde E,\tilde A)}$.
    \end{proof}

\section{Radiality index }
In this section we introduce a less well-known index, the \textit{radiality index}.

    \begin{defi}[\textbf{radiality index}]\hfill\\
        The \textit{radiality index} of $(E,A)$ is the smallest number $p_{\rm rad}^{(E,A)}\in \N_0$, such that there exists a $\omega \in \R$, $C>0$ with $(\omega, \infty)\subseteq \rho(E,A)$ and
        \begin{align}\label{rad-index}
            \begin{split}
                \left\Vert(\lambda_0 E-A)\inv E\cdot\ldots\cdot(\lambda_{p_{\rm rad}^{(E,A)}} E-A)\inv E\right\Vert &\leq C\prod_{k=0}^{p_{\rm rad}^{(E,A)}} \frac{1}{\left\vert\lambda_k-\omega\right\vert},\\
                \left\Vert E(\lambda_0 E-A)\inv \cdot\ldots\cdot E(\lambda_{p_{\rm rad}^{(E,A)}} E-A)\inv \right\Vert &\leq C\prod_{k=0}^{p_{\rm rad}^{(E,A)}} \frac{1}{\left\vert\lambda_k-\omega\right\vert}
            \end{split}
        \end{align}
        for all $\lambda_0,\ldots,\lambda_{p_{\rm{rad}}^{(E,A)}} >\omega$. The radiality index is called a \textit{complex radiality index}, denoted by $p_{\rm c, rad}^{(E,A)}$, if $\C_{\Real >\omega}\subseteq \rho(E,A)$ and \eqref{rad-index} holds for $p_{\rm{rad}}^{(E,A)}$.
    \end{defi}

    Clearly, $p_{\rm rad}^{(E,A)}\le p_{\rm c,rad}^{(E,A)}$. The radiality index $p=p_{\rm rad}^{(E,A)}\in\N_0$, originally known as \textit{weak $(E,p)$-radiality} \cite[p.~21]{fedorov-sviridyuk}, is not a commonly used  index  in finite-dimensional problems. However, the radiality index is one of the most useful index terms in infinite dimensions (see \cite{fedorov-sviridyuk, jacob-morris}).

    To get a better understanding for this definition one can start by looking at the radiality index with $p_{\rm rad}^{(E,A)}=0$, $\X=\Z$ and $E=I_{\X}$. Then, \eqref{rad-index} translates into the well known Hille-Yosida type estimate. This suggests that the radiality index implies further well-posedness results. In fact, as soon  as slightly stronger assumptions, namely  strong $(E,p)$-radiality, are considered, one obtains not only the well-posedness of the system, but also the existence of a Weierstraß-form \cite[Sec.~2.5 \&~2.6]{fedorov-sviridyuk}.

    \begin{prop}\label{rad-equiv}
        The radiality index is unique. In fact, let $(E,A) \sim (\tilde E, \tilde A)$. Then $p_{\mathrm{rad}}^{(E,A)}=p_{\mathrm{rad}}^{(\tilde E,\tilde A)}$ and $p_{\mathrm{c, rad}}^{(E,A)}=p_{\mathrm{c, rad}}^{(\tilde E,\tilde A)}$.
    \end{prop}

    \begin{proof}
        Since $(E,A)\sim (\tilde E, \tilde A )$ there exists two isomorphisms $P\colon \X \to \tilde \X$, $Q\colon \Z \to \tilde \Z$, such that
        \begin{equation*}
            E = Q\inv \tilde E P \quad \text{and} \quad A = Q\inv \tilde A P.
        \end{equation*}
        Assume, that $(E,A)$ has radiality index $p_{\rm rad}^{(E,A)}$. Thus, there exists positive constants $C>0$, $\omega >0$, such that $(\omega, \infty)\subseteq \rho(E,A)$ and
        \begin{align*}
            \left\Vert (\lambda_0 E-A)\inv E\cdot\ldots\cdot(\lambda_{p_{\rm rad}^{(E,A)}} E-A)\inv E \right\Vert & \leq \frac{C}{(\lambda_0-\omega)\cdot\ldots\cdot(\lambda_{p_{\rm rad}^{(E,A)}} -\omega)}, \\
            \left\Vert E (\lambda_0 E-A)\inv \cdot\ldots\cdot E(\lambda_{p_{\rm rad}^{(E,A)}} E-A)\inv \right\Vert & \leq \frac{C}{(\lambda_0-\omega)\cdot\ldots\cdot(\lambda_{p_{\rm rad}^{(E,A)}} -\omega)},
        \end{align*}
        for all $\lambda_0,\ldots,\lambda_{p_{\rm rad}^{(E,A)}}\in\rho(E,A)$. Since
        \begin{equation*}
            P (\lambda E-A)\inv Q\inv = (\lambda \tilde E-\tilde A)\inv
        \end{equation*}
        for all $\lambda \in \rho(E,A)$ we derive $(\omega,\infty) \subseteq \rho(E,A)\subseteq\rho(\tilde E, \tilde A)$ and therefore
        \begin{align*}
            \Vert (\lambda_0 \tilde E-\tilde A)\inv \tilde E\cdot\ldots&\cdot(\lambda_{p_{\rm rad}^{(E,A)}} \tilde E-\tilde A)\inv \tilde E\Vert
            \\&\leq \left\Vert P \right\Vert \Vert (\lambda_0 E-A)\inv E\cdot\ldots\cdot(\lambda_{p_{\rm rad}^{(E,A)}} E-A)\inv E \Vert \Vert P\inv  \Vert \\
            &\leq \frac{\tilde C}{(\lambda_0-\omega)\cdot\ldots\cdot(\lambda_{p_{\rm rad}^{(E,A)}} -\omega)},\\
            \Vert \tilde E(\lambda_0 \tilde E-\tilde A)\inv \cdot\ldots&\cdot  \tilde E (\lambda_{p_{\rm rad}^{(E,A)}} \tilde E-\tilde A)\inv \Vert \\
            &\leq \left\Vert Q \right\Vert \Vert E(\lambda_0 E-A)\inv \cdot\ldots\cdot E(\lambda_{p_{\rm rad}^{(E,A)}} E-A)\inv  \Vert \left\Vert Q\inv  \right\Vert \\
            &\leq \frac{\tilde C}{(\lambda_0-\omega)\cdot\ldots\cdot(\lambda_{p_{\rm rad}^{(E,A)}} -\omega)}
        \end{align*}
        for all $\lambda_0, \ldots, \lambda_{\rm rad}^{(E,A)}>\omega$, for $\tilde C \coloneqq 2C\max\{\left\Vert P \right\Vert,\left\Vert P\inv \right\Vert,\left\Vert Q \right\Vert,\left\Vert Q\inv \right\Vert\}$. Consequently, $(\tilde E, \tilde A)$ has at most radiality index $p_{\rm rad}^{(E,A)}$ . By an identical argument,  $(E,A)$ has at most radiality index $p_{\rm rad}^{(\tilde E,\tilde A)}$. The equality $p_{\mathrm{c, rad}}^{(E,A)}=p_{\mathrm{c, rad}}^{(\tilde E,\tilde A)}$ is shown in the same way.
    \end{proof}

    \begin{ex}{\bf (A system with radiality index 0.)
        }
        We consider the dynamics on an interval $[a,b]$  of an undamped beam fixed at one end and free at the other are given by the partial differential equations
        \begin{subequations}
            \begin{align}
                \rho\frac{\partial^2 v(\x,t)}{\partial t^2}&=\alpha\dz{}\dz{v(\x,t)}-\gamma\beta\dz{}\dz{p(\x,t)}\\
                \mu\frac{\partial^2 p(\x,t)}{\partial t^2}&=\beta\dz{}\dz{p(\x,t)}-\gamma\beta\dz{}\dz{v(\x,t)},
            \end{align}\label{eq:beam1}
        \end{subequations}
        where $\x\in[a,b[$, $ t>0$, $v(\x,t)$ is the longitudinal displacement and $p(\x,t)$ electric charge. The material parameters are material density $\rho>0, $ magnetic permeability $\mu\ge 0 ,$ elastic stiffness  $\alpha_1>0$, impermittivity  $\beta>0$, piezoelectric coefficients $\gamma>0 $ and we define
        \begin{equation*}
            \alpha=\alpha_1+\gamma^2\beta ,
        \end{equation*}
        One choice of boundary conditions is subject to the boundary conditions
        \begin{subequations}
            \begin{align}
                v(a,t)&=0,\\
                p(a,t)&=0,\\
                \beta\frac{\partial p(b,t)}{\partial \x}-\gamma\beta\frac{\partial v(b,t)}{\partial \x}&=0, \\
                \alpha \frac{\partial v(b,t)}{\partial \x}-\gamma\beta\frac{\partial p(b,t)}{\partial \x}&=0.
                \end{align}\label{eq:beamBC}
        \end{subequations}
        The total energy is
        \begin{align}
            \Ha(t)= \frac{1}{2} \int\limits_a^b \rho \left(\dt{v(\x,t)}\right)^2+\alpha_1 \left(\dz{v(\x,t)}\right)^2 + \mu  \left(\dt{p(\x,t)}\right)^2 + \beta \left(\dz{p(\x,t)}-\gamma \dz{v(\x,t)}\right)^2   d \x\label{eq:Hexam} .
        \end{align}
        This model was shown to be a well-posed port-Hamiltonian system associated with a contraction semigroup \cite{MO2013}. However, there are quite a few choices for the  state variables. We use here a different choice of state variable, suggested by  Hans Zwart (personal communication):
        \begin{equation*}
                \z(\x,t)\coloneqq \begin{bmatrix}
                z_1(\x,t)\\z_2(\x,t)\\z_3(\x,t) \\ z_4(\x,t)
            \end{bmatrix}=
            \begin{bmatrix}
                \dz{v(\x,t)}\\
                \sqrt{\rho} \dt{v(\x,t)}\\
                \dz{p(\x,t)}\\
                \sqrt{\mu} \dt{p(\x,t)}\\
            \end{bmatrix} \, .
        \end{equation*}
        Defining
        \begin{equation*}
            Q\coloneqq \begin{bmatrix}
                \alpha & 0&-\gamma \beta&0 \\
                    0& I & 0 &0\\
                    -\gamma \beta & 0 & \beta &0 \\
                    0 & 0 & 0& I
            \end{bmatrix}, \quad
            E\coloneqq
            \begin{bmatrix}
                \sqrt{\rho} & 0 & 0 & 0 \\
            0&\sqrt{\rho} & 0&0 \\
            0& 0 &\sqrt{\mu} & 0 \\
            0 & 0 & 0 & \sqrt{\mu}
            \end{bmatrix} ,\quad
            P_1 \coloneqq
            \begin{bmatrix}
            0 &I& 0&0\\
            I& 0 & 0 &0\\
            0 & 0 & 0 & I\\
            0& 0 & I&0
            \end{bmatrix},
        \end{equation*}
        the PDAE can be written
        \begin{equation*}
            \dt{}E \z(\x,t)=A \z(\x,t),
        \end{equation*}
        where $A=P_1\dz{} Q$.
        This choice of state variables yields equations in the standard pH form in \cite{jacob-zwart}.

        Generally, $\mu$ is very small and it is often taken to be zero, which yields the {\em quasi-static} piezo-electric beam.
        If $\mu = 0 $ the operator $E$ becomes singular and the fourth state variable becomes identically zero. The PDAE becomes
        \begin{equation*}
            \dt{}  \begin{bmatrix} \sqrt{\rho} & 0 & 0 & 0 \\
                0&\sqrt{\rho} & 0&0 \\
                0& 0 &0 & 0 \\
                0 & 0 & 0 & 0
            \end{bmatrix}
            \begin{pmatrix}
                \dz{v(\x,t)}\\
                \sqrt{\rho} \dt{v(\x,t)}\\
                \dz{p(\x,t)}\\
                0
            \end{pmatrix}
                =
                P_1\dz{} Q
            \begin{pmatrix}
                \dz{v(\x,t)}\\
                \sqrt{\rho} \dt{v(\x,t)}\\
                \dz{p(\x,t)}\\
                0
            \end{pmatrix} .
        \end{equation*}
        Removing the fourth column of $E,  Q$ since the 4th variable is zero,
        we obtain
        \begin{equation*}
                \dt{}
            \begin{bmatrix}
                \sqrt{\rho} & 0 & 0  \\
                0&\sqrt{\rho} & 0 \\
                0& 0 &0  \\
                0 & 0 & 0  \end{bmatrix}
            \begin{pmatrix}
                \dz{v(\x,t)}\\
                \sqrt{\rho} \dt{v(\x,t)}\\
                \dz{p(\x,t)}\\
            \end{pmatrix}
                =
            \begin{bmatrix}
                0 &I& 0&0\\
                I& 0 & 0&0  \\
                0 & 0 & 0&I \\
                0& 0 & I&0
            \end{bmatrix}
            \dz{}
            { \begin{bmatrix}
                \alpha & 0&-\gamma \beta \\
                0& I & 0 \\
                -\gamma \beta & 0 & \beta  \\
                0 & 0 & 0
            \end{bmatrix}  }
            \begin{pmatrix}
                \dz{v(\x,t)}\\
                \sqrt{\rho} \dt{v(\x,t)}\\
                \dz{p(\x,t)}
            \end{pmatrix} .
        \end{equation*}
        This PDAE is radial of degree zero; see \cite[sect. IV]{jacob-morris}.
\end{ex}

    Next, we compare the radiality index with the resolvent  index.

    \begin{prop}
            If  the (complex) radiality index $p_{\rm rad}^{(E,A)}$ exists, then the (complex) resolvent index also exists . Furthermore,
            $p_{\rm rad}^{(E,A)}+1\geq \pres^{(E,A)}$ ($p_{\rm c, rad}^{(E,A)}+1\geq p_{\rm c,res\vphantom{d}}^{(E,A)}$).
    \end{prop}

    \begin{proof}
        We have to show that there exists a $C>0$, $\omega>0$, such that $(\omega, \infty)\subseteq \rho(E,A)$ ($\C_{\Real >\omega}\subseteq \rho(E,A)$) and \eqref{def:resolvent-index} holds for all $\lambda \in(\omega,\infty)$ ($\lambda\in\C_{\Real >\omega}$) and a $p\leq p_{\rm{rad}}+1$ ($p\leq p_{\rm c,rad}+1$). The first part of this statement is implied by existence of the  radiality index. Hence, we only have to prove  \eqref{def:resolvent-index}. This has already been shown in \cite[Lem.~3.1.1]{fedorov-sviridyuk} for an $(E,p)$-sectorial operator pair $(E,A)$. However, since the definition of  $(E,p)$-sectorial introduced in \cite[sec.~3.1]{fedorov-sviridyuk} coincides with the radiality index except that \eqref{rad-index} holds for a sector $\left\{ \mu \in \C \,\middle|\, \left| \arg (\mu-\omega)\right| <\theta, \mu\neq \omega\right\}$ for a given $\omega\in \R$ and $\theta \in (\frac{\pi}{2},\pi)$, one can simply follow the proof of \cite[Lem.~3.1.1]{fedorov-sviridyuk} for positive $\lambda>\omega$ (for $\lambda\in\C_{\Real>\omega}$).
    \end{proof}



The following example shows that in general  existence of the resolvent index does not imply that the radiality index exists.

    \begin{ex} \label{ex:res-but-no-radiality}
        {\bf (A system where the resolvent index exists, but the radiality index does not.)} Let $\X=L^2(0,\infty) \times \R$ and define
        \begin{align*}
            E=\begin{bmatrix}I & 0\\ 0 & 0\end{bmatrix}, \quad A=\begin{bmatrix}\partial_x & 0\\\delta_0 & 1\end{bmatrix}
        \end{align*}
        with $\dom(A)\coloneqq H^1_0(0,\infty)\times \R$, where $\delta_0\colon L^2(0,\infty)\to\R, x\mapsto x(0)$. Let $\lambda >0$ and $\Big(\begin{smallmatrix}f \\ g\end{smallmatrix}\Big)\in \X$. Then
        \begin{align*}
            (\lambda E-A) \begin{pmatrix}x\\y\end{pmatrix} =\begin{pmatrix}f\\g\end{pmatrix}
        \end{align*}
        if and only if
        \begin{align*}
            \begin{pmatrix}x\\ y\end{pmatrix} = \begin{pmatrix}\int_{\cdot}^\infty \e^{\lambda (\cdot-s)}f(s)\dx[s]\\ -g-\int_0^\infty \e^{-\lambda s} f(s)\dx[s]\end{pmatrix}.
        \end{align*}
        Hence,
        \begin{align*}
            (\lambda E-A)\inv \begin{pmatrix}f\\ g\end{pmatrix} = \begin{pmatrix}\int_\cdot^\infty \e^{\lambda (\cdot-s)}f(s)\dx[s]\\-g -\int_0^\infty \e^{\lambda (0-s)}f(s)\dx[s] \end{pmatrix}
        \end{align*}
        and $\left\Vert (\lambda E-A)\inv\right\Vert\leq M = M \lambda^0$ for a $M>0$ for all $\lambda \geq 0$. Thus, $(E,A)$ has resolvent index $p_{\rm res}^{(E,A)} = 1$. Furthermore, we have
        \begin{align*}
            \left((\lambda E-A)\inv E\right)^{p+1} \begin{pmatrix}f\\ g\end{pmatrix} = \begin{pmatrix}
                \int_\cdot^\infty \left( \int_{s_{p+1}}^\infty \left( \ldots \left( \int_{s_2}^\infty\e^{\lambda (\cdot - s_1)} f(s_1)\dx[s_1]\right) \ldots\right)  \dx[s_p]\right) \dx[s_{p+1}]\\
                -\int_0^\infty \left( \int_{s_{p+1}}^\infty \left( \ldots \left( \int_{s_2}^\infty\e^{\lambda (0-s_1)} f(s_1)\dx[s_1]\right) \ldots\right)  \dx[s_p]\right) \dx[s_{p+1}]
            \end{pmatrix}.
        \end{align*}
        Let $f(\tau) = \tau \e^{-\lambda \tau}$. Then, $\left\Vert f \right\Vert_{L^2}=\frac{1}{\lambda ^{\frac{3}{2}}}$ and
        \begin{align}\label{eq:est}
            \begin{split}
                \left\Vert \left((\lambda E-A)\inv E\right)^{p+1} \right\Vert &\geq \left\Vert \left((\lambda E-A)\inv E\right)^{p+1}\begin{pmatrix}\lambda^{\frac{3}{2}} f \\ 0\end{pmatrix} \right \Vert \\
                &\geq \left| \int_0^\infty \left( \int_{s_{p+1}}^\infty \left( \cdots \left( \int_{s_2}^\infty\e^{-2\lambda s_1} s_1 \dx[s_1]\right) \right)  \dx[s_p]\right) \dx[s_{p+1}]\right| \\
                &=\frac{p+1}{2^{p+1}} \frac{1}{\lambda^{p+1-\frac{3}{2}}}
            \end{split}
        \end{align}
        for every $p\in \N$. Consequently, the radiality index does not exist. (Because if it exists, then there would exist $C>0$, $\omega >0$, such that
        \begin{align*}
            \left\Vert \left( (\lambda E-A)\inv E\right)^{p+1} \right\Vert \leq C \frac{1}{(\lambda-\omega)^{p+1}}
        \end{align*}
        for all $\lambda>\omega$, which would contradict \eqref{eq:est}.)
    \end{ex}

\section{Nilpotency index} \label{nilp-index}
Next, we are going to look at what is  probably the best known index,  the \textit{nilpotency index} (also known as the \textit{Weierstraß-index}). As the alternative term  suggests, the most important part of the definition is the existence of the Weierstraß form.

    \begin{defi}[\textbf{nilpotency index}]\hfill\\
        Assume, that the DAE \eqref{eqn:dae} has a Weierstraß form given by \eqref{eqn:form-inf}. Then, the \textit{nilpotency index} of $(E,A)$, denoted by $p_{\rm nilp}^{(E,A)}\in\N_0$, is the nilpotency degree of $N$ if it is present and $0$ if $N$ is absent. In the latter case one has $(E,A)\sim (I_{\Y^1},A_1)$ in  \eqref{eqn:form-inf}.
    \end{defi}
\begin{prop}\label{nilp-unique}
    Assume that the DAE \eqref{eqn:dae} has a Weierstraß form. Let $\lambda \in\C$ such that $\lambda E-A$ is bijective (which exists by our regularity assumption on \eqref{eqn:dae}). Then the nilpotency index of $(E,A)$ is
    is the smallest number $k\in\N$, such that
\[\ker \big((\lambda E-A)^{-1}E\big)^k=\ker \big((\lambda E-A)^{-1}E\big)^{k+1}.\]
In particular, the nilpotency index is well-defined.
\end{prop}
\begin{proof}
Assume that
 $P\colon \X \to \Y^1\times \Y^2$, $Q\colon \Z \to \Y^1\times \Y^2$, such that
 \[E = Q\inv \begin{bmatrix}I_{\Y^1} & 0\\0 & N\end{bmatrix} P, \quad  A = Q\inv \begin{bmatrix}A_1 & 0\\0 & I_{\Y^2}\end{bmatrix} P.\]
Then the result follows, since for all $k\in\N$,
\[\big((\lambda E-A)^{-1}E\big)^k=Q
\begin{bmatrix}(\lambda I_{\Y^1}-A_1)^{-k} & 0\\0 & (\lambda N-I_{\Y^2})^{-1}N^k\end{bmatrix}Q^{-1}.\qedhere
\]
\end{proof}

    According to this definition, the nilpotency index is always a 
    natural number (including $0$).
    This means that the nilpotency index can never be $\infty$, since we require a nilpotent operator $N$ in \eqref{eqn:form-inf}. To handle this more general situation, one would have to replace the nilpotency of $N$ with  \textit{quasi-nilpotency}, namely with $\sigma(N)=\{0\}$. For example, $N\colon \ell^2\to\ell^2$, $(x_1, x_2, , x_3,\ldots)\mapsto (0, \frac{x_1}{2^1}, \frac{x_2}{2^2}, \frac{x_3}{2^3},\ldots)$. 

    A disadvantage of the nilpotency index is that it  requires the Weierstraß form. In finite dimensions one only needs the system to be regular to obtain such a form,  see, for instance, \cite[Def.~2.9]{Kunkelmehrmann} for a constructive procedure. There is no standard procedure to obtain the Weierstraß form for infinite-dimensional systems, and, in fact, it has not been proven that such a form always exists.

   \begin{prop} \label{prop:nilp}
        \begin{enumerate}[(a)]
            \item \label{chain+1>=nilp}
                If the nilpotency index exists, then the chain index also exists with $p_{\rm nilp}^{(E,A)}= p_{\rm chain}^{(E,A)}$.
            \item \label{res>=nilp}
                If the nilpotency index and resolvent index exist, then $\pres^{(E,A)}\geq p_{\rm nilp}^{(E,A)}$.
            \item \label{rad+1>=nilp}
                If the nilpotency index and radiality index exist, then $p_{\rm rad}^{(E,A)}+1\geq p_{\rm nilp}^{(E,A)}$.
        \end{enumerate}
    \end{prop}

    \begin{proof}
        We start with the proof of Part (\ref{chain+1>=nilp}).
        Since $(E,A)\sim (\tilde E, \tilde A)\coloneqq \left(\Big[\begin{smallmatrix} I_{\Y^1} & 0\\0 & N\end{smallmatrix}\Big], \Big[\begin{smallmatrix} A_1 & 0\\0 & I_{\Y^2}\end{smallmatrix}\Big] \right)$ for a Hilbert space $\Y=\Y^1\times \Y^2$, where $N$ has nilpotency degree $p_{\rm nilp}^{(E,A)}$, there exists two isomorphisms $P\colon \X \to \Y^1\times \Y^2$, $x\mapsto (P_1x,P_2x)$, $Q\colon \Z \to \Y^1\times\Y^2$, $z\mapsto (Q_1 z, Q_2 z)$, with
        \begin{equation}\label{dae-equiv}
            E=Q\inv \begin{bmatrix} I_{\Y^1} & 0\\0 & N\end{bmatrix}  P \quad \text{and} \quad A = Q\inv \begin{bmatrix} A_1 & 0\\0 & I_{\Y^2}\end{bmatrix} P.
        \end{equation}
        For simplicity we will show, that $p_{\rm nilp}^{(\tilde E,\tilde A)} = p_{\rm chain}^{(\tilde E,\tilde A)}$ (because then the rest follows from Proposition \ref{chain-equiv} and \ref{nilp-unique}). Let $x_k=\left(\begin{smallmatrix} x_{k,1}\\ x_{k,2}\end{smallmatrix}\right)$, $k=1, \ldots, p$ denote an arbitrary chain of length $p-1$. Then $x_{k,1}=0$ for all $k=1,\ldots,p$ and
        \begin{equation*}
        \tilde E x_1 =0, \tilde E x_2 = \tilde A x_1, \quad \ldots \quad ,\tilde E x_p = \tilde A x_{p-1}.
        \end{equation*}
        Hence $Nx_{l,2} = x_{l-1,2}$ for all $2\leq l$ and $x_{1,2} = N^1x_{2,2} = \ldots = N^{p-1} x_{p,2}$ . Since $x_{1,2} \neq 0$, $p>p_{\rm nilp}^{(E,A)}$ is not possible. Thus, the existence of a Weierstraß form implies that the chain index also exists and that it is at most $p_{\rm nilp}^{(E,A)}$. In order to prove that these two index-terms are equal, we need to show that there does exists a chain of length $p_{\rm nilp}^{(E,A)}$. This follows directly by choosing $x_k\coloneqq \left(\begin{smallmatrix} x_{k,1}\\ x_{k,2}\end{smallmatrix}\right) = \left(\begin{smallmatrix} 0 \\ N^{p_{\rm nilp}^{(\tilde E,\tilde A)}-k}z\end{smallmatrix}\right)$, $k=1,\ldots, p_{\rm nilp}^{(\tilde E,\tilde A)}$, for a $z\in \Y^2$ such that $N^{p_{\rm nilp}^{(\tilde E,\tilde A)}}z=0$ and $N^{p_{\rm nilp}^{(\tilde E,\tilde A)}-1}z\neq 0$.

        We now prove Part (\ref{res>=nilp}). From the definition of the resolvent index, there exists $C>0, \omega >0$, such that $(\omega,\infty)\subseteq \rho(E,A)$ and
        \begin{equation*}
            \left\Vert(\lambda E-A)\inv\right\Vert\leq C \left\vert\lambda\right\vert^{\pres^{(E,A)}-1}.
        \end{equation*}
        By Proposition \ref{res-equiv} ,  for some $\tilde C>0$
        \begin{align}\label{eqn:1}
            \left\Vert \begin{bmatrix}(\lambda I_{\Y^1} - A_1)\inv & 0\\0&(\lambda N-I_{\Y^2})\inv\end{bmatrix}\right\Vert  \leq \tilde C \left\vert\lambda\right\vert^{\pres^{(E,A)}-1}
        \end{align}
        for all $\lambda>\omega$. Since $N$ is nilpotent with nilpotency degree $p_{\rm nilp}^{(E,A)}$ , $(\lambda N-I_{\Y^2})\inv = -\sum_{i=0}^{p_{\rm nilp}^{(E,A)}-1} (\lambda N)^i$ and together with \eqref{eqn:1} we derive
        \begin{align*}
            \left\Vert(\lambda N-I_{\Y^2})\inv\right\Vert  &= \lambda^{p_{\rm nilp}^{(E,A)}-1}\left\Vert I_{\X^2}\frac{1}{\lambda^{p_{\rm nilp}^{(E,A)}-1}} + N \frac{1}{\lambda^{p_{\rm nilp}^{(E,A)}-2}} + \ldots + N^{p_{\rm nilp}^{(E,A)}-2} \frac{1}{\lambda} + N^{p_{\rm nilp}^{(E,A)}-1} \right\Vert\\
            &\leq \tilde C \lambda^{\pres^{(E,A)}-1}
        \end{align*}
        for all $\lambda >\omega$ and therefore $p_{\rm nilp}^{(E,A)}\leq \pres^{(E,A)}$.

        Finally, in order to prove Part (\ref{rad+1>=nilp}) we assume that the radiality index $p_{\rm rad}^{(E,A)}$ exists. That is, there exists  $C>0$ and  $\omega>0$ such that
        \begin{align*}
            \left\Vert ((\lambda N - I_{Y^2})\inv N)^{p_{\rm rad}^{(E,A)}+1} \right\Vert & \leq \left\Vert ((\lambda \tilde E - \tilde A)\inv \tilde E)^{p_{\rm rad}^{(E,A)}+1}\right\Vert\\
            & \leq \left\Vert P\inv \right\Vert  \left\Vert ((\lambda E-A)\inv E)^{p_{\rm rad}^{(E,A)}+1}\right\Vert \left\Vert P \right\Vert \\
            & \leq \frac{C}{(\lambda-\omega)^{p_{\rm rad}^{(E,A)}+1}}
        \end{align*}
        for all $\lambda >\omega$. Now, $((\lambda N-I_{\Y^2})\inv N)^{p_{\rm rad}^{(E,A)}+1}=\frac{1}{\lambda^{p_{\rm rad}^{(E,A)}+1}}\left(-\sum_{k=1}^{p_{\rm nilp}^{(E,A)}-1} (\lambda N)^k\right)^{p_{\rm rad}^{(E,A)}+1}$ is a polynomial of degree $1\leq \rm{deg}\leq (p_{\rm nilp}^{(E,A)}-1)$. Thus,
            \begin{equation}\label{eq:estimate}
                \left\Vert  \left(-\sum_{k=1}^{p_{\rm nilp}^{(E,A)}-1} (\lambda N)^k\right)^{p_{\rm rad}^{(E,A)}+1}\right\Vert \leq C \frac{\lambda^{p_{\rm rad}^{(E,A)}+1}}{(\lambda-\omega)^{p_{\rm rad}^{(E,A)}+1}}.
            \end{equation}
       The left-hand-side contains terms of the form $(\lambda N)^n$ where $n \geq p_{\rm rad}^{(E,A)}+1. $ Since it is bounded by   the right-hand-side of \eqref{eq:estimate} which is bounded as $\lambda \to \infty , $ $p_{\rm nilp}^{(E,A)}\leq p_{\rm rad}^{(E,A)}+1$, as was to be proven.
    \end{proof}


    \begin{remark}
        For non-negative and self-adjoint $E\in\L(\X)$ and dissipative $A\colon\dom(A)\subseteq\X\to\X$ the nilpotency index and the chain index is at most $2$. This follows directly from Theorem \ref{res-index-bound} and Proposition \ref{prop:nilp}.
    \end{remark}



    \begin{ex} {\bf (A system with $p_{\rm rad}^{(E,A)}= 1$ and $p_{\rm nilp}^{(E,A)}=2$)}
        Let us recall Example \ref{never-ending-example}. Defining
         \begin{align*}
            R_{1,s}&=(sI-A_o)^{-1} (I-B G(s)^{-1} C (sI-A_o)^{-1}) ,\\  R_{2,s}&=G(s)^{-1} C (sI-A_o)^{-1},\\ L_{2,s}&=(sI-A_o)^{-1} B G(s)^{-1}
        \end{align*}
        one has
        \begin{align*}
            (sE-A)^{-1} E &= \bem R_{1,s} & 0 \\  R_{2,s} &0  \enm , \\[1ex] E (sE-A)^{-1}  &= \bem R_{1,s} & L_{2,s} \\ 0 &0 \enm .
        \end{align*}

        As a particular example, let $\W = L^2 (0,1)$,
        \begin{equation}
            A_o z = z^{\prime \prime} , \quad D(A_o) = \{w \in H^2 (0,1) \mid z^\prime (0) =z^\prime (1) = 0 \} , \quad b=c= \mathds{1} .
            \label{heat-eg}
        \end{equation}
        With these definitions of $A_o$, $B$ and $C$,
        \begin{align*}
            (sI-A_o)\inv Bu &= \frac{1}{s}\mathds{1}u,\\
            C(sI-A_o)\inv z &= \frac{1}{s}\langle z, \mathds{1} \rangle,\\
            G(s) &= C(sI-A_o)\inv B = \frac{1}{s}
        \end{align*}
        and so  $R_{1,s}z=(sI-A_o)\inv z - \mathds{1}\langle z,1\rangle \frac{1}{s}$, $R_{2,s} z = \langle z,\mathds{1}\rangle$, $L_{2,s} u = \mathds{1}u$.
      Since $R_{2,s}$ is independent of $s$, the radiality degree must be larger than $0$.

      Define the projection onto $\W_1 \coloneqq \ker C \subset \W$,
	\begin{align*}
            Q_c z &\coloneqq z -\frac{\la z , c \ra}{\la c,c\ra} \, c .
        \end{align*}
        Then, with $\W_2 \coloneqq  {\rm span  \, } c$, $Q_c$ splits $\W$ into $\W_1 \oplus  \W_2$.
        It is easy to see that for $\alpha \mathds{1}\in {\rm span \,} c$  $R_{1,s} \alpha \mathds{1}=0$ and for $z \in \ker C$  $R_{1,s}z=(sI-A_o)\inv z\in \ker C$. Thus, for $w = \alpha z\mathds{1}\in W_1\oplus W_2$
        \begin{align*}
            (sE-A)\inv E(\mu E-A)\inv E \begin{pmatrix} z+\alpha\mathds{1} \\ u  \end{pmatrix} &= \bem R_{1,s}R_{1,\mu} & 0 \\  R_{2,s}R_{1,\mu} &0  \enm \begin{pmatrix} z+\alpha\mathds{1} \\ u  \end{pmatrix} = \begin{pmatrix} R_{1,s}R_{1,\mu}z
            \\ 0  \end{pmatrix}, \\
            E(sE-A)\inv E(\mu E-A)\inv \begin{pmatrix} z+\alpha\mathds{1} \\ u  \end{pmatrix} &= \bem R_{1,s}R_{1,\mu} & R_{2,s}L_{2,\mu}  \\ 0 & 0\enm \begin{pmatrix} z+\alpha\mathds{1} \\ u  \end{pmatrix} = \begin{pmatrix} R_{1,s}R_{1,\mu}z  \\ 0  \end{pmatrix}.
        \end{align*}
        Since $A_o$ is the generator of a contraction semigroup,
        \begin{align*}
            \left\Vert R_{1,s}R_{1,\mu}\right\Vert \leq \frac{1}{s\mu},\quad s,\mu>0.
        \end{align*}
        Hence, the radiality index $(E,A)$ is  $1$.

        Furthermore, it is possible to rewrite $(E,A)$ into a Weierstraß form with $p_{\rm nilp}^{(E,A)}=2$.
       For more concise notation,  define $\tilde C u \coloneqq \frac{1}{\langle c,c \rangle} cu .$

       We can define the isomorphisms $U\colon \W\times \C \to \W_1 \times\W_2\times \C$ and $V\colon \W_1 \times \W_2\times \C \to \W \times \C$
        \begin{equation*}
            U \coloneqq
            \bem
            Q_c & 0 \\[1.5ex]
                0 & \tilde C\\[1.5ex]
                       C & 0
            \enm, \qquad
            V \coloneqq
            \bem
                I & I & 0\\
                     0 & 0 & I
            \enm
        \end{equation*}
        which have inverses
        \begin{equation*}
            U^{-1} \coloneqq
            \bem
                I & A_o & B\\[1.5ex]
                0 & C & 0
            \enm, \qquad
            V^{-1} \coloneqq
            \bem
                Q_c & 0 \\
                I-Q_c & 0\\
                       0 & I
            \enm \, .
        \end{equation*}
        These mappings will be applied to \eqref{eq-zero1} to obtain a splitting of the system  into equations on $\W_1 \times \W_2 \times \C$.
        Noting that
        \begin{itemize}
            \item
            if $z \in \W_1, $ $A_o z \in \W_1 ,$
            \item $\tilde C C = I - Q_c, $ $=0 $ on $\ker C$, $I$ on span    $c$
            \item $C B=1,$
        \end{itemize}
        the isomorphisms $U$ and $V$  lead to
        \begin{align*}
            \tilde E&\coloneqq U\bem I&0 \\ 0&0\enm V
            &  \tilde A & \coloneqq U\bem A_o&B \\ C&0\enm V\\
               &=  \bem Q_c &Q_c &0 \\ 0&0&0 \\ \tilde C & \tilde C&0 \enm  \, , \;& &=\bem  Q_cA_o-Q_c A_o \tilde C C & Q_cA_o-Q_c A_o \tilde C C&0 \\\tilde C C &\tilde C C&0 \\ \tilde C(A_o) &\tilde C A_o&1 \enm \, \\
            &=  \bem I &0 &0 \\ 0&0&0 \\ 0 & \tilde C&0 \enm  \, & &=  \bem  A_o &0&0 \\ 0&I&0 \\ 0&0&1 \enm  .
        \end{align*}
        Define
        \begin{equation*}
            N= \bem 0&0 \\ \tilde C&0 \enm,\qquad
            \tilde A_o z=z^{\prime \prime}
        \end{equation*}
        with domain $D(A_o) \cap \ker C,$ which is dense in $\ker C $ \cite{Morris_rebarber_2007}.
        The operator $A_o$  will generate a $C_0$-semigroup on $\W_1 = \ker C .$
        Via the isomorphisms $U$ and $V$ the system \eqref{eq-zero1} is equivalent to the Weierstraß form
        \begin{equation}
            \frac{{\rm d}}{{\rm d}t} \bem I_{\W_1} & 0 \\ 0 & N \enm \begin{pmatrix} z_1 (t)\\ z_2 (t) \end{pmatrix} = \bem \tilde A_o & 0 \\ 0 & I_{\W_2 \times \mathbb C} \enm  \begin{pmatrix} z_1 (t)\\ z_2 (t) \end{pmatrix}.
            \label{eq-zero1-wei}
        \end{equation}
        A simple calculation shows that $N^2 =0$, and thus, the nilpotency index of the system is $2$.

        Systems of the form \eqref{eq-zero1} are known as Hessenberg index 2 systems in the finite-dimensional situation, see for example, \cite{LamourMaerzTischendorf2013}.


        This problem is identical to the question of establishing the zero dynamics of a system with state-space  realization $(A_o, B,C)$. That is, finding the largest space $\W_1 $ on which for initial conditions in $\W_1$ the dynamics remain in $\W_1$.
        Clearly $\W_1 \subseteq \ker C$. In this case, the largest space is $\W_1 = \ker C $. A  more general situation with $c \in D(A_o^*)$ and $ b \neq c$ is described  in \cite{Zwart-book}. The case where $c \notin D(A_o^*) $ is treated in \cite{Morris_rebarber_2007}.
           If $c \notin D(A_o^*)$, in general the zero dynamics are only associated with    an integrated semigroup \cite{Morris_rebarber_2007}.
           A similar type of construction was shown to hold for a class of boundary control systems with collocated observation; that is, the unbounded generalization of $\la b,c\ra\neq 0$ in \cite{ReisSelig} and for diffusion problems on an interval in \cite{Byrnes2006}.
    \end{ex}

    \begin{ex}\label{ex1}
    {\bf (A system where the nilpotency index exists, but resolvent index does not.)}
        Let $A_0=-\partial^2_x+ix$ be the complex Airy operator on $\X^1=L^2(\R)$ with
        \begin{equation*}
            \dom(A_0)=\left\{u \in H^2(\R) \,\middle|\, xu\in L^2(\R)\right\}.
        \end{equation*}
        By \cite{Helffer, ArnalSiegl2022}
        we have $\sigma(A_0)=\emptyset$ and there exists a constant $C>0$, such that
        \begin{equation*}
            \left\Vert(\lambda I_{\X^1}-A_0)\inv\right\Vert= C (\Real \lambda)^{-\frac{1}{4}} \e^{\frac{4}{3}(\Real \lambda)^{\frac{3}{2}}}
        \end{equation*}
        for all $\Real \lambda >0$. Thus, $\left\Vert(\lambda I_{\X^1}-A_0)\inv\right\Vert\to\infty$ for $\Real \lambda \to\infty$.
        Let $N\colon \X^2\to\X^2$ be a nilpotent operator of degree $p\in \N$. Then
        \begin{equation*}
            \frac{{\rm d}}{{\rm d}t}
            \underbrace{\begin{bmatrix}I_{\X^1}&0\\0&N\end{bmatrix}}_{=E}
            \begin{pmatrix}x_1\\x_2\end{pmatrix} =
            \underbrace{\begin{bmatrix}A_0&0\\0&I_{\X^2}\end{bmatrix}}_{=A}
            \begin{pmatrix}x_1\\x_2\end{pmatrix}
        \end{equation*}
        defines a DAE with nilpotency index $p=p_{\rm nilp}^{(E,A)}$, but since $\left\Vert(\lambda I_{\X^1}-A_0)\inv\right\Vert$ grows exponentially the resolvent index does not exist.
    \end{ex}

    \begin{ex} {\bf (A system where the nilpotency index exists, but radiality index does not.)}
        Example \ref{ex1} shows that in general the existence of the nilpotency index does not imply the existence of the radiality index.
    \end{ex}

    \begin{ex} {\bf (A system where the resolvent index exists, but nilpotency index does not.)}
Consider the system
\begin{align}
\tfrac{\partial}{\partial t}x_1(t,\xi)&=-\tfrac{\partial}{\partial \xi}x_1(t,\xi),\label{eq:DAEex1}\\
0&=-x_1(t,0),\label{eq:DAEex2}\\
0&=-x_1(t,1)+x_2(t),\label{eq:DAEex3}\\
\tfrac{\rm d}{{\rm d} t}x_2(t)&=x_3(t),\label{eq:DAEex4}
\end{align}
which corresponds to a~differential-algebraic equation $\frac{{\rm d}}{{\rm d}t}Ex(t)=Ax(t)$ with
\begin{equation}
\begin{aligned}
\X&=L^2([0,1])\times\C^2,\quad \dom(A)=H^1([0,1])\times\C^2,\\
\Z&=L^2([0,1])\times\C^3,
\end{aligned}\label{eq:difftransportsp}
\end{equation}
and
\begin{equation}E=\begin{bmatrix}I_{L^2}&0&0\\0&0&0\\0&0&0\\0&1&0\end{bmatrix},\quad A=\begin{bmatrix}-\frac{\partial}{\partial \xi}&0&0\\-\delta_0&0&0\\-\delta_1&1&0\\0&0&1\end{bmatrix},\label{eq:difftransportop} 
\end{equation}
where $\delta_\xi\in H^1([0,1])^*$ is the evaluation operator at $\xi\in[0,1]$.
The resolvent fulfills, for all $x_1,x_2,x_2\in\C$, $x_2\in L^2([0,1])$,
\[(\lambda E-A)^{-1}\begin{pmatrix}x_1\\x_2\\x_3\\x_4\end{pmatrix}=\begin{pmatrix}{\rm e}^{-s}x_3+\int_0^1{\rm e}^{-s(1-\xi)}f(\xi)d\xi\\s{\rm e}^{-s}x_3+s\int_0^1{\rm e}^{-s(1-\xi)}f(\xi)d\xi-x_2\\{\rm e}^{-s\cdot}x_3+\int_0^\cdot{\rm e}^{-s(\cdot-\xi)}f(\xi)d\xi\end{pmatrix},\]
which gives $\pres^{(E,A)}=2$. It has however been shown in \cite{ReiTis05} that the system has no nilpotency index.
\end{ex}

As we have already seen through Example \ref{ex:res-but-no-radiality} and \ref{ex1} the nilpotency index, the resolvent index and the radiality index do not have to coincide. This is different in finite dimensions as we will see next.

    \begin{prop}\label{nilp>=rad+1}
        Let $\X$ and $\Z$ be finite-dimensional. Then $p_{\rm nilp}^{(E,A)}=p_{\rm rad}^{(E,A)}+1=\pres^{(E,A)}$.
    \end{prop}
    \begin{proof}
        In \cite[Lem.~2.1]{GernandHallerReis} it is shown that $p_{\rm nilp}^{(E,A)}=\pres^{(E,A)}$. Thus it remains to show $p_{\rm nilp}^{(E,A)}=p_{\rm rad}^{(E,A)}+1$.

        By Proposition \ref{prop:nilp} (\ref{rad+1>=nilp}) we already have $p_{\rm rad}^{(E,A)}+1\geq p_{\rm nilp}^{(E,A)}$. Let $\left(\big[\begin{smallmatrix}I & 0\\ 0 & N\end{smallmatrix}\big], \big[\begin{smallmatrix}J & 0\\0&I\end{smallmatrix}\big]\right)$ be the Weierstraß form of $(E,A)$. For $\lambda \in \rho(E,A)$ the right-resolvent of this form is
            \begin{equation}\label{right-resolvent}
                R_\lambda\coloneqq \left(\lambda \begin{bmatrix}I&0\\0&N\end{bmatrix} - \begin{bmatrix}J&0\\0&I\end{bmatrix}\right)\inv\begin{bmatrix}I&0\\0&N\end{bmatrix} = \begin{bmatrix}(\lambda I-J)\inv & 0\\ 0 & (\lambda N - I)\inv N\end{bmatrix}
            \end{equation}
        and since $N$ commutes with $(\lambda N-I)\inv=-\sum_{k=0}^{p_{\rm nilp}^{(E,A)}-1}(\lambda N)^k$ this is the same as the left-resolvent. Let $\lambda_0,\ldots,\lambda_{p_{\rm nilp}^{(E,A)}-1}\in\rho(E,A)$. Using the nilpotency of $N$ we get
            \begin{equation}
                (\lambda_0 N -I)\inv N\cdot\ldots\cdot(\lambda_{p_{\rm nilp}^{(E,A)}-1} N-I)\inv N = (\lambda_0 N-I)\inv \cdot\ldots\cdot (\lambda_{p_{\rm nilp}^{(E,A)}-1} N-I)\inv N^{p_{\rm nilp}^{(E,A)}}=0.
            \end{equation}
        Thus,
            \begin{equation*}
                R_{\lambda_0}\cdot\ldots\cdot R_{\lambda_{p_{\rm nilp}^{(E,A)}-1}}=\begin{bmatrix}\displaystyle\prod_{k=0}^{p_{\rm nilp}^{(E,A)}-1}(\lambda_k I-J)\inv & 0\\ 0 & 0\end{bmatrix}.
            \end{equation*}
        Since $J$ is in Jordan form there exists a $C>0$, such that for all $\lambda_0,\ldots,\lambda_{p_{\rm nilp}^{(E,A)}-1}>\omega$ we have
            \begin{equation}
                \left\Vert(\lambda_0 I -J)\inv\cdot\ldots\cdot(\lambda_{p_{\rm nilp}^{(E,A)}-1}I-J)\inv\right\Vert \leq K \frac{1}{(\lambda_0-\omega)\cdot\ldots\cdot(\lambda_{p_{\rm nilp}^{(E,A)}-1}-\omega)}.
            \end{equation}
        Thus, $(E,A)$ has at most radiality index $p_{\rm nilp}^{(E,A)}-1$ and therefore $p_{\rm nilp}^{(E,A)}\geq p_{\rm rad}^{(E,A)}+1$.
    \end{proof}

\section{Differentiation index}
The differentiation index is based on taking formal derivatives up to a~certain order which leads to an ordinary differential equation. For details, see, for instance, \cite[Sec.~3.10]{LamourMaerzTischendorf2013} and \cite[Sec.~3.3]{Kunkelmehrmann}. Here we present a~generalization of this concept to abstract differential-algebraic equations of the form
\eqref{eqn:dae}
%
with the restriction  that only the case where $f\equiv0$ is considered. An
extension to nonzero $f$ is a  subject for future research.
Using the fact that  the operators $E$ and $A$ do not depend on time, the $k$-th formal derivative of \eqref{eqn:dae} with $f=0$ is
\[\frac{{\rm d}}{{\rm d}t}E\frac{d^k}{dt^k}x(t)=A\frac{d^k}{dt^k}x(t).\]
The collection of the first $\mu$ formal derivatives of \eqref{eqn:dae} is
\begin{equation}\label{eq:derarray}
\underbrace{\begin{bmatrix}
A&-E\\
&A&-E\\
&&\ddots&\ddots\\
&&&A&-E
\end{bmatrix}}_{\eqqcolon M_\mu}
\begin{pmatrix}
    x(t)\\\frac{{\rm d}}{{\rm d}t}x(t)\\\vdots\\\frac{d^{\mu}}{dt^{\mu}}x(t)\\\frac{d^{\mu+1}}{dt^{\mu+1}}x(t)
\end{pmatrix}
=0
\end{equation}
which will be referred to as a {\em derivative array of order
$\mu$}. The operator $M_\mu$ maps from $\dom(A)^{\mu+1}
\times\X$ to $\Z^{\mu+1}$.

\begin{defi}[\textbf{Differentiation index}]\label{eq:diffind}\hfill\\
Let $\X,\Z$ be Banach spaces, and let a~linear differential-algebraic equation
$\frac{{\rm d}}{{\rm d}t}Ex(t)=Ax(t)$ be given, where $E\in\L(\X,\Z)$,
$(A,\dom(A))$ is a closed and densely defined linear
operator from $\X$ to $\Z$. 
The DAE $\frac{{\rm d}}{{\rm d}t}Ex(t)=Ax(t)$ has differentiation index $p_{\rm diff}^{(E,A)}\in\N_0$, if the following holds:
\begin{enumerate}[(a)]
\item There exists a~Banach space $\tilde{\X}$, such that $\tilde{\X}$ is densely embedded into $\X$, an operator $(S,\dom(S))$
which is the generator of a~strongly continuous semigroup on $\tilde{\X}$,
and the operator $M_{p_{\rm diff}^{(E,A)}}$ as defined in \eqref{eq:derarray} satisfies
\[
\begin{pmatrix}
x_0\\x_1\\\vdots\\x_{p_{\rm diff}+1}\end{pmatrix}\in \ker M_{p_{\rm diff}^{(E,A)}}\quad
\Longrightarrow\,x_0\in\dom(S)\text{ with } x_1=Sx_0.
\]
\item The number $p_{\rm diff}^{(E,A)}$ is minimal with the properties in (a).
\end{enumerate}
We call
\begin{equation}\label{eq:complode}
    \frac{{\rm d}}{{\rm d}t}x(t)=Sx(t)
\end{equation}
an {\em abstract completion ODE
of \eqref{eqn:dae}}.
\end{defi}
Before establishing a connection between the solutions of the abstract completion ODE and the DAE itself, we introduce a lemma that characterizes solutions.
This result is a~direct consequence of the integration by parts formula. It is therefore omitted.
\begin{lem}\label{lem:test}
Let $\X,\Z$ be Banach spaces, and let a~linear differential-algebraic equation
$\frac{{\rm d}}{{\rm d}t}Ex(t)=Ax(t)$ be given, where $E\in \L(\X,\Z)$,
$(A,\dom(A))$ is a closed and densely defined linear
operator from $\X$ to $\Z$.
\begin{enumerate}[(a)]
    \item If $x\colon[0,\infty)\to X$ solves $\frac{{\rm d}}{{\rm d}t}Ex(t)=Ax(t)$, then
for all test functions $\varphi\colon[0,\infty)\to\R$ (i.e., those which are smooth and have compact support in $(0,\infty)$),
\begin{equation}-\int_{\R_\geq 0} \frac{{\rm d}}{{\rm d}t}\varphi(t)Ex(t)dt=\int_{\R_\geq 0} \varphi(t)Ax(t)dt.\label{eq:test}\end{equation}
\item Conversely, if $x\colon [0,\infty)\to \R$ fulfills \eqref{eq:test} for all test functions $\varphi\colon[0,\infty)\to\R$, and it is continuously differentiable with $x(t)\in\dom(A)$ for all $t\ge0$, then it is a solution of $\frac{{\rm d}}{{\rm d}t}Ex(t)=Ax(t)$.
    \end{enumerate}
\end{lem}

\begin{thm}
Let $\X,\Z$ be Banach spaces, and let $E\in \L(\X,\Z)$, and
$(A,\dom(A))$ be a closed and densely defined linear
operator from $\X$ to $\Z$. 
Assume that the linear differential-algebraic equation $\frac{{\rm d}}{{\rm d}t}Ex(t)=Ax(t)$ has differentiation index $p_{\rm diff}^{(E,A)}$.
If $x$ solves $\frac{{\rm d}}{{\rm d}t}Ex(t)=Ax(t)$, then it fulfills the abstract completion ODE \eqref{eq:complode}.
\end{thm}

\begin{proof}
Assume that $x$ solves $\frac{{\rm d}}{{\rm d}t}Ex(t)=Ax(t)$. Lemma~\ref{lem:test} implies that for all test functions $\varphi\colon [0,\infty)\to\R$
\[-\int_{\R_\geq 0} \frac{{\rm d}}{{\rm d}t}\varphi(t)Ex(t)dt=\int_{\R_\geq 0} \varphi(t)Ax(t)dt.\]
Now testing with $(-1)^k\frac{d^k}{dt^k}\varphi(t)$, $k=0,\ldots,p_{\rm diff}^{(E,A)}+1$, we obtain that the vectors
\begin{align*}
    x_k&=(-1)^k\int_{\R_\geq 0}\frac{d^k}{dt^k}\varphi(t)x(t)dt, \quad k=0,\ldots,p_{\rm diff}^{(E,A)}+1,
\end{align*}
satisfy
\[\begin{bmatrix}
A&-E\\
&A&-E\\
&&\ddots&\ddots\\
&&&A&-E
\end{bmatrix}
\begin{pmatrix}
    x_0\\x_1\\\vdots\\x_{p_{\rm diff}^{(E,A)}}\\x_{p_{\rm diff}^{(E,A)}+1}
\end{pmatrix}
=-0,
\]
and the definition of the differentiation index yields
\[x_1=Sx_0.
\]
Thus, for all test functions $\varphi\colon [0,\infty)\to\R$
\begin{equation*}
    -\int_{\R_\geq 0} \tfrac{{\rm d}}{{\rm d}t}\varphi(t)x(t)dt=\int_{\R_\geq 0} \varphi(t)Sx(t)dt.
\end{equation*}
Another application of Lemma~\ref{lem:test}  now shows that $x$ is a~solution of the  abstract completion ODE of \eqref{eqn:dae}.
\end{proof}
 \begin{remark}
 Let $\X,\Z$ be Banach spaces, and let $E\in \L(\X,\Z)$, and
$(A,\dom(A))$ be a closed and densely defined linear
operator from $\X$ to $\Z$. 
Assume that the linear differential-algebraic equation $\frac{{\rm d}}{{\rm d}t}Ex(t)=Ax(t)$ has differentiation index $p_{\rm diff}^{(E,A)}$, and let \eqref{eq:complode}
be an abstract completion ODE.
Let $x_0\in\X$, such that the differential-algebraic initial value problem $\frac{{\rm d}}{{\rm d}t}Ex(t)=Ax(t)$, $Ex_0=Ex(0)$ has a~solution $x$. Since the solution of the abstract completion boundary control system  \eqref{eq:complode}  with initial value $x(0)=x_0\in\X$ is unique, we can conclude that $x$ is the unique solution of the differential-algebraic initial value problem $\frac{{\rm d}}{{\rm d}t}Ex(t)=Ax(t)$.
\end{remark}
\begin{ex} {\bf (A system with $p_{\rm diff}^{(E,A)}=2$)} Once again, consider the system \eqref{eq:DAEex1}--\eqref{eq:DAEex4},
which corresponds to a~differential-algebraic equation $\frac{{\rm d}}{{\rm d}t}Ex(t)=Ax(t)$ with
spaces \eqref{eq:difftransportsp}
and operators as in \eqref{eq:difftransportop}.\\
The system does not have differentiation index zero, since $E$ has a~nontrivial nullspace. To analyze whether it has differentiation index one,
assume that $x_1(0,\cdot)=x_{10}$ for some given $x_{10}\in L^2([0,1])$. Then the solution of the initial-boundary value problem \eqref{eq:DAEex1}, \eqref{eq:DAEex2} reads $x_1\colon [0,\infty)\to L^2([0,1])$ with
\[(x_1(t))(\xi)=x_1(t,\xi)=\begin{cases}
    x_0(\xi-t):&\xi>t,\\
    0:&\xi<t.
\end{cases}\]
By thereafter plugging this into \eqref{eq:DAEex2} and \eqref{eq:DAEex3}, we obtain that
\[x_2(t)=\begin{cases}
    x_0(1-t):&t\leq 1,\\
    0:&t>1,
\end{cases}\]
and \eqref{eq:DAEex4} yields that $x_3(t)=\frac{d}{d t}x_2(t)$. This requires differentiability of $x_2$,
which is fulfilled, if $x_0\in H^3([0,1])$ with $x_0(0)=x_0'(0)=x_0''(0)=0$. In this case,
\[x_3(t)= \begin{cases}
    -\frac{\partial}{\partial \xi}x_{10}(1-t):&t\leq 1,\\
    0:&t>1.
\end{cases}\]
Hence, a~classical solution of \eqref{eq:DAEex1}--\eqref{eq:DAEex4} exists if
\[x_{10}\in H^3([0,1]) 
\text{ with } x_{10}(0)=x_{10}'(0)=x_{10}''(0)=0.\]
Next we analyse the differentiation index $p_{\rm diff}^{(E,A)}\in\N$:
Since
with
\[x_1=\left(\begin{smallmatrix}0\\0\\1\end{smallmatrix}\right),\,x_2=\left(\begin{smallmatrix}0\\1\\0\end{smallmatrix}\right)\in \X,\]
\[\left(\begin{smallmatrix}{0}\\x_{1}\\x_{2}\end{smallmatrix}\right)\in \ker M_1\]
one has $p_{\rm diff}^{(E,A)}>1$.

It will now be shown that  $p_{\rm diff}^{(E,A)}=2 .$
To this end, let $x_k\in \X$, $k=0,1,2,3$, partitioned as
\[x_k=\left(\begin{smallmatrix}x_{k1}\\x_{k2}\\x_{k3}\end{smallmatrix}\right), \quad x_{k1}\in L^2(0,1),\; x_{k2},x_{k3}\in\C,\]
and
\begin{equation}
    \left(\begin{smallmatrix}x_{0}\\x_{1}\\x_{2}\\x_3\end{smallmatrix}\right)\in \ker M_2.\label{eq:x012kerM2}
\end{equation}
Then
\[\tfrac{\partial}{\partial \xi}x_{11}=-x_{21},\quad
\tfrac{\partial}{\partial \xi}x_{21}=-x_{21},\]
and
\[0=\delta_0 x_{11},\quad 0=\delta_0 x_{21}=-\delta_0 \tfrac{\partial}{\partial \xi}x_{11},\]
which implies that
\[x_{11}\in \tilde{\X}_1\coloneqq \setdef{x\in H^2(0,1)}{x(0)=x'(0)=0}.\]
Analogously, we see that
\[x_{01}\in \tilde{\mathcal{Y}}_1\coloneqq \setdef{x\in H^3(0,1)}{x(0)=x'(0)=x''(0)=0}.\]
Further, by
\begin{align*}
x_{11}&=-\tfrac{\partial}{\partial \xi}x_{10},\\
x_{12}&=x_{03},\\
x_{13}&=x_{22}=\delta_1x_{21}=-\delta_1\tfrac{\partial}{\partial \xi}x_{11}=\delta_1\tfrac{\partial^2}{\partial \xi^2}x_{10},
\end{align*}
we have that \eqref{eq:x012kerM2} implies that, for
\[\tilde{\X}\coloneqq \tilde{\X}_1\times \C^2\]
and $(S,\dom(S))$ with
\[\dom(S)=\tilde{\mathcal{Y}}_1\times\C^2,\quad S\begin{pmatrix}x_{01}\\x_{02}\\x_{03}\end{pmatrix}=\begin{pmatrix}-\frac{\partial}{\partial \xi}x_{01}\\x_{03}\\x_{10}''(1)\end{pmatrix}.\]
The operator $S$ is indeed the generator of a~strongly continuous semigroup on $\tilde{\X}$, namely $T(\cdot)$ with
\[T(t)\begin{pmatrix}x_{01}\\x_{02}\\x_{03}\end{pmatrix}=\begin{pmatrix}\mathbb{S}_t x_{01}\\x_{02}+tx_{03}+\int_0^t\int_0^\tau\big(\mathbb{S}_s x_{10}\big)(1)dsd\tau \\x_{03}+\int_0^t\big(\mathbb{S}_\tau x_{10}\big)(1)d\tau \end{pmatrix},\]
where $\mathbb{S}_t$ denotes the right shift of length $t$ on functions defined on the interval $[0,1]$.\\
In the case where the initial condition fulfills
\[x_{02}=x_{01}(1),\quad x_{03}=-x_{01}'(1),\]
it can be seen that
\[x(t)=T(t)\begin{pmatrix}x_{01}\\x_{02}\\x_{03}\end{pmatrix}\]
is indeed a~solution of \eqref{eq:DAEex1}-\eqref{eq:DAEex4}.
\end{ex}

\begin{prop}
\begin{enumerate}[(a)]
\item If the differentiation index $p_{\rm diff}^{(E,A)}$ exists, then the chain index exists  and $p_{\rm chain}^{(E,A)}\le p_{\rm diff}^{(E,A)}$.
\item If the nilpotency index $p_{\rm nilp}^{(E,A)}$ exists and the operator $A_1$ in the Weierstra\ss\ form generates a $C_0-$semigroup, then the differentiation index exists and $p_{\rm nilp}^{(E,A)}= p_{\rm diff}^{(E,A)}$.
\end{enumerate}
\end{prop}

\begin{proof}
\begin{enumerate}[(a)]
\item   Assume that the  differentiation index exists. Further, let $(x_1,\ldots,x_p)$, $p\in\N$, be a chain of $(E,A)$. Then
    \[M_{p-1}\left(\begin{smallmatrix}0\\x_1\\\vdots\\x_p\end{smallmatrix}\right)=0.\]
    This implies that there does not exist a~mapping $S\colon \dom(S)\to \X$ such that every $z\in \ker M_{p-1}$ satisfies
    \[z=\left(\begin{smallmatrix}x_0\\Sx_0\\x_2\\\vdots\\x_p\end{smallmatrix}\right)\]
for some $x_0\in\X$, $x_2,\ldots,x_p\in \dom(A)$.
Consequently, $p\leq p_{\rm diff}^{(E,A)};$ that is, the
chain index of $(E,A)$ does not exceed the
differentiation index of $(E,A)$.
\item Assume that $(E,A)$ is in Weierstra\ss\ form \eqref{eqn:form-inf} with $\nu = p_{\rm nilp}^{(E,A)}$ and $N^\nu=0$, $N^{\nu-1}\ne0$.
Let $x_k\in \X= \Y^1\times\Y^2$, $k\in\N$, be partitioned as
\[x_k=\left(\begin{smallmatrix}x_{k,1}\\x_{k,2}\end{smallmatrix}\right), \quad x_{k,1}\in \Y^1,\; x_{k,2}\in\Y^2.\]
Using the Weierstra\ss\ form, it follows that
\begin{equation}\left(\begin{smallmatrix}x_{0}\\\vdots\\x_{k+1}\end{smallmatrix}\right)\in \ker M_k
\label{eq:Mkker}\end{equation}
if, and only if, for some $x_{1,0}\in\dom(A_1^{k+1})$, $x_{2,0}\in\Y^2$,
\begin{equation}x_{1,i}=A_1^i x_{1,0},\quad
N^{i}x_{2,i}=x_{2,0},\quad i=1,\ldots,k+1.\label{eq:Mkker2}
\end{equation}
Let $k<\nu$. Then by choosing some $z\in\Y^2$ with $N^{\nu-1}z\neq0$, we obtain \eqref{eq:Mkker} with
\[x_0=\left(\begin{smallmatrix}0\\0\end{smallmatrix}\right),\, x_1=\left(\begin{smallmatrix}0\\N^{\nu-1}z\end{smallmatrix}\right),\ldots, \, x_{k+1}=\left(\begin{smallmatrix}0\\N^{\nu-k-1}z\end{smallmatrix}\right).\]
Since $x_0=0$ and $x_1\neq0$, an operator $S$ with the properties as described in Definition~\ref{eq:diffind} cannot exist.

On the other hand, the equivalence between \eqref{eq:Mkker} and \eqref{eq:Mkker2} implies that  if $k=\nu$
\[x_1=\left(\begin{smallmatrix}A_1x_{0,1}\\0\end{smallmatrix}\right).\]
Consequently, we can choose $\tilde{\X}=\X$
and $S\colon \dom(S)\subset\X\to\X$ with $\dom(S)=\dom(A_1)\times \Y^2$ and
\[S=\begin{bmatrix}A_1&0\\0&0\end{bmatrix},\]
which is the generator of the strongly continuous semigroup
\[T(\cdot)=\begin{bmatrix}T_1(\cdot)&0\\0&I_{\Y^2}\end{bmatrix},\]
where $T_1(\cdot)$ is the semigroup generated by $A_1$. This shows that the differentiation index coincides with the nilpotency index of $N$.\qedhere
\end{enumerate}
\end{proof}

\section{Perturbation index}

Finally, we define the  perturbation index for DAEs on possibly infinite-dimensional spaces. As the name indicates, the  perturbation index is a measure of sensitivity of solutions with respect to perturbations of the  problem.

    \begin{defi}[\textbf{perturbation index}]\hfill\\
        Consider the DAE \eqref{eqn:dae} on $[0,T]$ with input $\delta\in C^{p_{\rm pert}^{(E,A)}-1}([0,T],\Z)$  and  solution $x:$
         \begin{align}\label{eqn:dae.pert}
                \frac{\text{d}}{\text{d}t} Ex(t)=Ax(t)+\delta(t), \quad t\geq 0.
            \end{align}
             The  \textit{perturbation index $p_{\rm pert}^{(E,A)}$} is the smallest number $p\in\N$ such that there exists a constant $c>0$ satisfying the following for all $\delta:$
            \begin{align}
              \max_{0\le t\le T} \|x(t)\|_{\X} &\le c\left( \|x(0)\|_{\X} + \sum_{i=0}^{p-1}
\max_{0\le t\le T}  \left\|\frac{\text{d}^{i}}{\text{d}t^{i}}\delta(t)\right\|_{\Z}
                    \right),&&\text{if }p>0, \\                \max_{0\le t\le T} \|x(t)\|_{\X} &\le c\left( \|x(0)\|_{\X} +\int_0^T\|\delta(t)\|_{\Z}{\rm d}t  \right)        ,&&\text{if }p=0.\label{p=0}
            \end{align}

    \end{defi}

First we show that the perturbation index is identical for equivalent systems.
\begin{prop}\label{equ=pert}
Consider  two differential-algebraic systems $\frac{{\rm d}}{{\rm d}t}Ex=Ax$, $\frac{{\rm d}}{{\rm d}t}\tilde E\tilde x=\tilde A\tilde x$ with $(E, A) \sim (\tilde E, \tilde A)$. Then
\[p_{\rm pert}^{(E,A)} =p_{\rm pert}^{(\tilde E, \tilde A)}.\]
\end{prop}
\begin{proof}
Let $P\colon\X\to\tilde\X$, $Q\colon\Z\to\tilde\Z$,
be bounded isomorphisms as in Definition~\ref{def:equiv}. Then the result follows, since $x$ is a~solution of \eqref{eqn:dae.pert}, if, and only if, $\tilde x=Px$ is a~solution of $\frac{{\rm d}}{{\rm d}t}\tilde E\tilde x=\tilde A\tilde x+\tilde\delta$ for $\tilde\delta=Q^{-1}\delta$.
\end{proof}

    \begin{prop}\label{nilp=pert}
The following holds:
        \begin{enumerate}[(a)]
            \item \label{pert-index-a} If the nilpotency index $p_{\rm nilp}^{(E,A)}$ and the perturbation index $p_{\rm pert}^{(E,A)}$ exist, then $p_{\rm nilp}^{(E,A)} \leq p_{\rm pert}^{(E,A)}$.
            \item \label{pert-index-b} If the nilpotency index $p_{\rm nilp}^{(E,A)}$ exists and $A_1$ in \eqref{eqn:form-inf} generates a $C_0$-semigroup, then the perturbation index also exists and $p_{\rm nilp}^{(E,A)}=p_{\rm pert}^{(E,A)}$.
        \end{enumerate}
    \end{prop}

    \begin{proof}
        By Proposition~\ref{equ=pert}, it is no loss of generality to assume that the system is in Weierstra\ss\ form, that is,
            \[(E,A)= \left(\begin{bmatrix}I_{\Y^1} & 0\\0 & N\end{bmatrix},\begin{bmatrix}A_1 & 0\\0 & I_{\Y^2}\end{bmatrix}\right) \]
        and $p\coloneqq p_{\rm nilp}^{(E,A)}$.

        First, we show Part (\ref{pert-index-a}). If $p_{\rm nilp}^{(E,A)}=0$, then nothing has to be shown. If $p_{\rm nilp}^{(E,A)}=1$, then $N=0$ and thus, for the second part of a solution $x$ of \eqref{eqn:dae.pert} holds $x_2(t) =-\delta_2(t)$. Consequently, the maximum of $x$ cannot be estimated by the integral of $\|\delta_2\|$, and thus also an estimate of the form \eqref{p=0} is not possible. Thus, we have $p_{\rm pert}^{(E,A)}\le 1$ in this case.\\
        Finally, we assume that $p\coloneqq p_{\rm nilp}^{(E,A)}\ge 2$. Seeking for a~contradiction, we assume that $p_{\rm pert}^{(E,A)}<p$. Then there exists some $c>0$, such that for all solutions $x$ of \eqref{eqn:dae.pert},
        it holds
        \begin{equation}
            \max_{0\le t\le T} \|x(t)\|_{\X} \le c\left( \|x(0)\|_{\X} + \sum_{i=0}^{p-2}
            \max_{0\le t\le T}  \left\|\frac{\text{d}^{i}}{\text{d}t^{i}}\delta(t)\right\|_{\Z}
                                \right)
            \label{eq:contrad_est}
        \end{equation}
        Let $v_2\in\Y_2$ with $N^{p-1}v_2\neq0$.
        We choose $\delta_1=0$ $x_1(0)=0$ and $\delta_{2,n}\colon [0,T]\to\Y_2$ with
        \[\delta_{2,n}(t)=\frac{{\rm Re}\big(\imath^p{\rm e}^{\imath nt}\big)}{ n^{p-1}}v_2.\]
Then, by using that ${\rm Re}\big(\imath^p{\rm e}^{\imath nt}\big)$ is a~trigonometric function, we have
        \begin{equation}\label{eq:pertest1}
            \forall\,n\ge \frac{\pi}{T}:\quad\max_{0\leq t\leq T}\|N^{p-1}\delta_{2,n}^{(p-1)}\|\geq \|N^{p-1}v_2\|>0,
        \end{equation}
        and, further
        \begin{align}\label{eq:pertest2}
            &\forall\,k=1,\ldots,p-2:\quad\lim_{n\to\infty}\max_{0\leq t\leq T}\|\delta_{2,n}^{(k)}\|=0,\\
            &\delta_{2,n}(0)=
            {\rm Re}\left.\big(\imath{\rm e}^{\imath nt}\big)\right|_{t=0}v_2=0.\label{eq:pertest3}
        \end{align}
        Let $\delta_n=\left(\begin{smallmatrix}
            0\\\delta_{2,n}
        \end{smallmatrix}\right)$.
        Since  if  $x_{1,n}=0$ the solution $x_n=\left(\begin{smallmatrix}
            x_{1,n}\\x_{2,n}
        \end{smallmatrix}\right)$ of $\frac{\text{d}}{\text{d}t} Ex_n(t)=Ax_n(t)+\delta_n(t)$ has $x_{1,n}(0)=0$   and also
            \[x_{2,n}=- \sum_{i=0}^{p-1} N^i \frac{\text{d}^i}{\text{d}t^i} \delta_{2,n}(t),\]
        we obtain from \eqref{eq:pertest1} and \eqref{eq:pertest2}  that
            \[\liminf_{n\to\infty}\max_{0\leq t\leq T}\|x_n(t)\|>0.\]
        On the other hand, by \eqref{eq:pertest1} and \eqref{eq:pertest3}, we obtain
            \[\lim_{n\to\infty}\underbrace{\|x_n(0)\|_{\X}}_{=0} + \sum_{i=0}^{p-2}
            \max_{0\le t\le T}  \left\|\frac{\text{d}^{i}}{\text{d}t^{i}}\delta_n(t)\right\|_{\Z}=0,\]
        which contradicts \eqref{eq:contrad_est}.

        Next, we will show Part (\ref{pert-index-b}).
        Let $x$ be a solution of DAE \eqref{eqn:dae.pert}, and partition
            \[x(t)=\begin{pmatrix}x_{1}(t)\\x_{2}(t)\end{pmatrix},\quad \delta(t)=\begin{pmatrix}\delta_1(t)\\\delta_2(t)\end{pmatrix}\]
        according to the  structure of the Weierstra\ss\ form.
        Then
        \begin{align}
            \label{eqn:dae.part1}
            \frac{\text{d}}{\text{d}t}  x_{1}(t)) &= A_1 x_{1}(t) +\delta_1(t), && t\geq 0,\\
            \label{eqn:dae.part2}
            N\frac{\text{d}}{\text{d}t} x_{2}(t) &= x_{2}(t) +\delta_2(t), && t\geq 0.
        \end{align}
        Let $S$ be the semigroup generated by  $A_1$.
        Then \eqref{eqn:dae.part1} yields
        \begin{align*}
           x_{1}(t) &=
            S(t) x_{1}(0)
            + \int_0^t
            S(t-\tau)\delta_1(\tau)\,\text{d}\tau.
        \end{align*}
        Since $S$ is bounded on $[0,T]$, we have $c_0\coloneqq \sup_{t\in[0,T]}\|S(t)\|<\infty$, $c_1=\max\{c_0,c_0T\}$, and we can conclude
        \begin{align}
            \label{eqn:est.part1}
            \max_{t\in[0,T]}
            \|x_{1}(t)\|_{\Y_1} &\le
            c_0\|x_{1}(0)\|_{\Y_1}
            + Tc_0 \max_{0\le t\le T} \|\delta_1(t)\|_{\Y_1}\\
           &\le  c_1\big(\|x_{1}(0)\|_{\Y_1}
            +  \max_{0\le t\le T} \|\delta_1(t)\|_{\Y_1}\big).
        \end{align}
        In the case where $p=0$, we have $\X=\Y_1$, $x=x_{1}$ and $\delta=\delta_{1}$, and we obtain from \eqref{eqn:est.part1} that $p_{\rm pert}^{(E,A)}=0$. If this is not the case, we find from \eqref{eqn:dae.part2} that
        \begin{align*}
            x_{2}(t) &= N \frac{\text{d}}{\text{d}t} x_{2}(t)- \delta_2(t)
            = N \frac{\text{d}}{\text{d}t} \left(N \frac{\text{d}}{\text{d}t} x_{2}(t)- \delta_2(t)\right) -  \delta_2(t)\\
            &= N^2 \frac{\text{d}^2}{\text{d}t^2} x_{2}(t)) - N \frac{\text{d}}{\text{d}t}  \delta_2(t)
            +\delta_2(t)
            = N^2 \frac{\text{d}^2}{\text{d}t^2} \left(N \frac{\text{d}}{\text{d}t} x_{2}(t)+\delta_2(t)\right) - N \frac{\text{d}}{\text{d}t} \delta_2(t)
            - \delta_2(t)\\
            &= N^3 \frac{\text{d}^3}{\text{d}t^3} x_{2}(t) - N^2 \frac{\text{d}^2}{\text{d}t^2} \delta_2(t) - N \frac{\text{d}}{\text{d}t}  \delta_2(t)
            +\delta_2(t)
            = \hdots
            = - \sum_{i=0}^{p-1} N^i \frac{\text{d}^i}{\text{d}t^i} \delta_2(t).
        \end{align*}
       Boundedness of $E$ yields that $N$ is bounded, and thus $c_2\coloneqq\max_{i=0}^{p-1} \|N^i\|<\infty$. Consequently,
        \begin{align}
            \label{eqn:est.part2}
            \max_{0\le t\le T} \|x_{*2}(t)\|_{ \X_2}
            \le c_2\left(  \sum_{i=0}^{p-1}
            \max_{0\le t\le T}  \left\|\frac{\text{d}^{i}}{\text{d}t^{i}} \delta_2(t)\right\|_{\tilde\Z_2}
            \right).
        \end{align}
        Collecting the bounds \eqref{eqn:est.part1} and \eqref{eqn:est.part2} implies that,
        setting $c\coloneqq\max\{c_1,Tc_1,c_2\},$
        \begin{align*}
            \max_{0\le t\le T} \|x(t)\|_{\X}
            \le c\left( \| x(0)\|_{\X}
            + \sum_{i=0}^{p-1}
            \max_{0\le t\le T}  \left\|\frac{\text{d}^{i}}{\text{d}t^{i}} \delta(t)\right\|_{\Z}
            \right) \, .
        \end{align*}
         Thus,  $p_{\rm pert}^{(E,A)}\leq p_{\rm nilp}^{(E,A)}$.
     \end{proof}
\begin{remark}
In \cite{Rei07}, a class of linear infinite-dimensional differential-algebraic systems was examined, which do not necessarily have a nilpotency index. It was revealed that a~perturbation analysis necessitates the consideration of stronger norms for the initial value. These norms may correspond to a type of spatial perturbation index for systems governed by partial differential-algebraic equations.
\end{remark}

 \section{Summary}
 The purpose of this paper was, as explained in the introduction, to review the common indices for DAEs and to extend them to infinite-dimensional DAEs where this has not already been done. Rigorous definitions were given for the nilpotency, resolvent, radiality, chain, differentiation and perturbation indices.  The definitions of the differentiation and perturbation indices are as far as we know, new in the Banach space setting. Unlike finite-dimensional DAEs, the existence of an index is not guaranteed, and furthermore the various indices  are not equivalent for general DAEs. Each index reveals different information about the system. We  were able to show  that in some cases existence of a particular index will imply existence of some other other indices. One implication is that the chain index is the least restrictive index since existence of any of the  other indices implies existence of the chain index, as well as a bound on the chain index.

 This work has revealed a number of open questions for DAEs on Banach spaces. Under what conditions does a particular index exist? If two indices exist, when are they equal? Which indices are more useful for the study of DAEs? It is hoped that this paper will inspire research on these and related questions.

\end{document}